\documentclass[11pt]{amsart}

\usepackage[usenames,dvipsnames]{color}
\usepackage{latexsym,xspace,enumerate,amsfonts,amsmath,amssymb,amscd}
\usepackage{xypic,tikz,hyperref,syntonly,MnSymbol,slashed,paralist}
\usepackage{marginnote}
\reversemarginpar
\usepackage[capitalize]{cleveref}
\usepackage[arrow, matrix, curve]{xy}

\newcommand{\N}{\mathbb{N}}
\newcommand{\Z}{\mathbb{Z}}

\newcommand{\R}{\mathbb{R}}
\newcommand{\C}{\mathbb{C}}

\newcommand{\GL}{\mr{GL}}
\newcommand{\SL}{\mr{SL}}

\newcommand{\U}{\mr{U}}
\newcommand{\SU}{\mr{SU}}

\newcommand{\su}{\mathfrak{su}}
\renewcommand\P{\mathbb{P}}

\newcommand{\mf}{\mathfrak}
\newcommand{\mr}{\mathrm}
\newcommand{\mb}{\mathbf}
\newcommand{\mc}{\mathcal}

\newcommand{\End}{\mathop{\rm End}\nolimits}

\newcommand{\Id}{\mathop{\rm Id}\nolimits}

\renewcommand{\ker}{\mathop{\rm ker}\nolimits}

\newcommand{\dom}{\operatorname{dom}}

\renewcommand{\Re}{\mathop{\rm Re}\nolimits}
\renewcommand{\Im}{\mathop{\rm Im}\nolimits}
\renewcommand{\deg}{\mathop{\rm deg}\nolimits}
\newcommand{\rk}{\mathop{\rm rk}\nolimits}
\newcommand{\Tr}{\mathop{\rm Tr}\nolimits}

\newcommand{\ad}{\mathop{\rm ad}\nolimits}

\newcommand{\del}{\partial}
\newcommand{\delb}{\bar{\partial}}

\newcommand{\note}[1]{\marginpar{\raggedright\if@twoside\ifodd\c@page\raggedleft\fi\fi\sf\scriptsize \red{RMK: #1}}}

\newcommand\red[1]{\textcolor{red}{#1}}

\newcommand{\be}{\begin{equation}}
\newcommand{\ben}{\begin{equation}\nonumber}
\newcommand{\ee}{\end{equation}}
\newcommand{\bp}{\begin{para}}
\newcommand{\ep}{\end{para}}
\newcommand{\RR}{\mathbb R}
\newcommand{\CC}{\mathbb C}

\newcommand{\fid}{\mathrm{fid}}

\newcommand{\supp}{\mathrm{supp}\,}
\newcommand{\calC}{{\mathcal C}}

\newcommand{\calM}{{\mathcal M}}
\newcommand{\calO}{{\mathcal O}}

\newcommand{\appr}{\mathrm{app}}
\newcommand{\exterior}{\mathrm{ext}}
\newcommand{\interior}{\mathrm{int}}
\newcommand{\sell}{\mathfrak{sl}}

\def\pr{\textrm{pr}}

\def\U{\mathrm{U}}
\def\GL{\mathrm{GL}}
\def\Stab{\operatorname{Stab}}
\def\D{\mathbb{D}}

\newtheorem{proposition}{\textbf{Proposition}}
\newtheorem{lemma}[proposition]{\textbf{Lemma}}
\newtheorem{corollary}[proposition]{\textbf{Corollary}}
\newtheorem{theorem}[proposition]{\textbf{Theorem}}

\theoremstyle{definition}
\newtheorem{definition}{\textbf{Definition}}

\newtheorem*{example*}{\textbf{Example}}
\newtheorem*{remark}{\textbf{Remark}}
\theoremstyle{remark}      
\newtheorem*{rem*}{Remark}
\newtheorem{step}{Step}

\newcounter{para}[section]
\newenvironment{para}[2][]{\refstepcounter{para}\noindent\ignorespaces{\bf #1\thepara. #2.} \rmfamily}{\noindent\ignorespacesafterend\bigskip}

\numberwithin{proposition}{section}
\numberwithin{definition}{section}
\begin{document}
\title{Ends of the moduli space of Higgs bundles}
\date{\today}

\author{Rafe Mazzeo}
\address{Department of Mathematics, Stanford University, Stanford, CA 94305 USA}
\email{mazzeo@math.stanford.edu}

\author{Jan Swoboda}
\address{Mathematisches Institut der LMU M\"unchen\\Theresienstra{\ss}e 39\\D--80333 M\"unchen\\ Germany}
\email{swoboda@math.lmu.de}
\address{Max-Planck-Institut f\"ur Mathematik\\Vivatsgasse 7\\D--53111 Bonn\\ Germany}
\email{swoboda@mpim-bonn.mpg.de}

\author{Hartmut Wei\ss}
\address{Mathematisches Institut der LMU M\"unchen\\ Theresienstra{\ss}e 39\\ D--80333 M\"unchen\\ Germany}
\email{weiss@math.lmu.de}

\author{Frederik Witt} 
\address{Mathematisches Institut der Universit\"at M\"unster\\ Einsteinstra{\ss}e 62\\ D--48149 M\"unster\\ Germany}
\email{frederik.witt@uni-muenster.de}

\thanks{RM supported by NSF Grant DMS-1105050, JS supported by DFG Grant Sw 161/1-1}

\maketitle

\begin{abstract}
We associate to each stable Higgs pair $(A_0,\Phi_0)$ on a compact Riemann surface $X$ a singular limiting configuration $(A_\infty,\Phi_\infty)$, assuming that $\det \Phi$ has only simple zeroes. We then prove a desingularization theorem by constructing a family of 
solutions $(A_t,t\Phi_t)$ to Hitchin's equations which converge to this limiting configuration as $t \to \infty$. This provides a new proof, via
gluing methods, for elements in the ends of the Higgs bundle moduli space and identifies a dense open subset of the 
boundary of the compactification of this moduli space. 
\end{abstract}

%
%
%
\section{Introduction}
The moduli space of solutions to Hitchin's equations on a compact Riemann surface
occupies a privileged position at the cross-roads of gauge-theoretic geometric
analysis, geometric topology and the emerging field of higher Teichm\"uller theory. 
These are equations for a pair $(A,\Phi)$, where $A$ is a unitary connection on a Hermitian vector
bundle $E$ over a Riemann surface $X$, and $\Phi$ an $\End(E)$-valued $(1,0)$-form (the `Higgs field'). 
We will mostly be concerned with the fixed determinant case, i.e.\ we consider only connections which 
induce a fixed connection on the determinant line bundle of $E$ and trace-free Higgs fields. Then the equations read
\begin{equation}
\begin{array}{rl}
&F_A^\perp + [ \Phi \wedge \Phi^*]  = 0 \\[0.5ex]
&\bar{\partial}_A \Phi = 0.
\end{array}
\label{he1}
\end{equation}
Here $F_A^\perp$ is the trace-free part of the curvature of $A$, which is a 2-form with values in the skew-Hermitian endomorphisms 
of $E$, and $\Phi^*$ is computed with respect to the Hermitian metric on $E$. We always assume that $X$ is compact below, and
we also assume that the genus of $X$ is bigger than $1$. 

The initial motivation for these is that, when $\Phi$ is the trivial rank $2$ bundle, 
they are the two-dimensional reduction of the standard self-dual Yang-Mills system, i.e.,
from $X \times \RR^2$ to $X$. However, these equations make sense for higher rank 
nontrivial bundles, and have also been studied when $X$ is a higher dimensional
K\"ahler manifold~\cite{si88,si92}. Beyond this initial presentation, they can also be studied by 
more purely algebraic and topological methods in terms of representations of (a central extension of) the fundamental group of $X$ into the Lie group $\SL(r,\C)$, $r = \mathrm{rk}(E)$ (see~\cite{go12} and references therein).  

In his initial paper on these equations \cite{hi87}, Hitchin established the existence of
a unique solution of these equations in the complex gauge orbit of any given initial
pair $(A_0, \Phi_0)$ which satisfy a stability condition slightly weaker than the standard
slope stability condition for $E$ alone. He went on to prove that the 
moduli space of solutions $\calM$ enjoys many nice properties. In particular, when $\mathrm{rk}(E) = 2$
and the degree of $E$ is odd, then $\calM$ is a smooth manifold of dimension $12 \gamma - 12$, where $\gamma$ is the genus of $X$. 
(In other cases it is at least a quasi-projective variety, but we shall focus on this simplest setting.) 
Furthermore, it has a natural hyperk\"ahler metric $g$ of Weil-Petersson type, with respect to which it is complete.
 In the intervening years, much has been learned about
its topology and many other features. However, surprisingly little is known about the metric structure at infinity.

In the past few years, however, a very intriguing conceptual picture has emerged through the work of Gaiotto, 
Moore and Neitzke~\cite{gmn10}. As part of a much broader picture concerning hyperk\"ahler metrics 
on holomorphic integrable systems, they describe a decomposition of the natural metric $g$ on $\calM$ 
as a leading term (the semiflat metric in the language of~\cite{fr99}) plus an asymptotic series of non-perturbative
corrections, which decay exponentially in the distance from some fixed point. The coefficients of these 
correction terms are given there in terms of a priori divergent expressions coming from a wall-crossing formalism.  

A further motivation is Hausel's result about the vanishing of the image of compactly supported cohomology in the 
ordinary cohomology~\cite{ha99}. In analogy with Sen's conjecture about the $L^2$-cohomology of the monopole moduli spaces~\cite{se94}, 
he conjectured further that the $L^2$-cohomology of the Higgs bundle moduli space must vanish. Partial confirmation of 
this conjecture were established shortly afterwards by Hitchin~\cite{hi00} who showed that the $L^2$-cohomology is
concentrated in the middle degree. Closely related results about $L^2$-cohomology of gravitational instantons, and partial 
confirmation of Sen's conjecture, were obtained by Hausel, Hunsicker and the first author \cite{HHM}. These papers suggest
that results of this type about $L^2$-cohomology rely on a better understanding of the metric asymptotics on $\calM$. 
 
One other recent development is contained in the recent pair of papers by Taubes \cite{ta1, ta2}. His
setting is for a closely related gauge theory on three-manifolds with gauge group $\SL(2,\C)$, but 
he notes there that his results transfer simply (and presumably with fewer technicalities) to the 
case of surfaces. He proves a compactness theorem
for those equations focusing on the specific problems caused by the noncompactness of the underlying 
group ($\SL(2,\CC)$ rather than $\SU(2)$). More specifically, he is able to deduce information about 
limiting behavior of solutions which diverge in a specific way in the moduli space.  While the results in
our paper are partly subsumed by those of Taubes, we hope that the constructive perspective adopted
here will be of value in the various types of questions described above. 

We can now describe our work and the results in this paper. Our initial motivation was to reach
a more detailed understanding of the structure of the space $\calM$ near its asymptotic boundary,
with the hope of using this to obtain information about the structure of the metric $g$ there. 
We do this by, in essence, reproving Hitchin's result for solutions which lie sufficiently far
out in $\calM$. We make here the simplifying assumption that the Higgs field $\Phi$ is {\em simple}, in the
sense that its determinant has simple zeroes. This implies, in particular, the stability (and thus the 
simplicity) of the pair $(E,\Phi)$ in the sense of Hitchin, but has also further technical implications. We first consider a family of `limiting 
configurations', consisting of certain singular pairs $(A_\infty, \Phi_\infty)$ which satisfy a decoupled 
version of Hitchin's equations, namely
\[
F^\perp_{A_\infty}=0,\quad[\Phi_\infty\wedge\Phi^*_\infty]=0,\quad\delb_{A_\infty}\Phi_\infty=0.
\]
Thus each $A_\infty$ is projectively flat with simple poles, while the limiting Higgs fields are 
holomorphic with respect to these connections and have a specified behavior near these poles. 
\begin{theorem}[Existence and deformation theory of limiting configurations]
Let $(A_0,\Phi_0)$ be any pair such that $q:= \det \Phi_0$ has only simple zeroes. Then there exists a complex gauge transformation $g_\infty$ on $X^\times=X\setminus q^{-1}(0)$ which transforms $(A_0,\Phi_0)$ into 
a limiting configuration. Furthermore, the space of limiting configurations with fixed determinant $q \in H^0(K_X^2)$ having simple zeroes is a torus of dimension $6\gamma-6$, where $\gamma$ is the genus of $X$.
\end{theorem}

We also consider the family of desingularizing `fiducial solutions' which will be used to `round off' the singularities in these 
limiting configurations. These fiducial solutions are an explicit family of radial solutions on $\CC$, the existence of which 
was pointed out to us by Neitzke, but since there does not seem to be an easily available reference for them in the literature, 
we provide a fairly complete derivation of their properties here. 

With these two types of components, we now pursue a standard strategy to construct exact solutions. Namely, we construct 
families of approximate solutions, which lie in the gauge orbit of some $(A, t\Phi)$ for $t$ large, and then use the linearization 
of a relevant elliptic operator to correct these approximate solutions to exact solutions. This yields the
\begin{theorem}[Desingularization theorem]
If $(A_\infty,\Phi_\infty)$ is a limiting configuration, then there exists a family $(A_t,\Phi_t)$ of solutions of the rescaled Hitchin equation
\[
F_A^\perp+t^2[\Phi\wedge\Phi^*]=0,\quad\delb_A \Phi=0
\]
provided $t$ is sufficiently large, where
\[
(A_t,\Phi_t) \longrightarrow (A_\infty,\Phi_\infty)
\]
as $t \nearrow \infty$, locally uniformly on $X^\times$ along  with all derivatives, at an exponential rate in $t$. Furthermore, $(A_t,\Phi_t)$ 
is complex gauge equivalent to $(A_0,\Phi_0)$ if $(A_\infty,\Phi_\infty)$ is the limiting configuration associated with $(A_0,\Phi_0)$.
\end{theorem}

In particular, we obtain Hitchin's existence theorem for pairs $(A,t\Phi)$ when $t$ is large and $\det \Phi$ has simple zeroes. The advantage of this method is that we obtain 
precise estimates on the shape of these solutions. This mirrors precisely what is obtained in \cite{ta1, ta2}, and it is not hard to deduce from this 
that the Weil-Petersson metric $g$ does indeed decompose as a principal term (essentially given by the deformation theory of the limiting configurations) 
and an exponentially decreasing error.  While we are able to capture the correct exponential rate, our method at present includes an extra polynomial factor,
so in particular we are not yet able to say anything about the leading coefficient of the first decaying term. 

To understand the entire end of the moduli space $\calM$ (when the degree of $E$ is odd), one must also consider 
non-simple Higgs fields. When the simplicity condition fails, the desingularizing fiducial solutions must be replaced 
by some more complicated special solutions. These new fiducial solutions are being studied in the ongoing thesis work 
of Laura Fredrickson, and it is expected that these gluing methods will adapt readily to incorporate her `multi-pole' fiducial solutions. 

\bigskip

\centerline{\textbf{Acknowledgements}}

\medskip
This project was initiated and conducted in the research SQuaRE ``Nonlinear analysis and special geometric structures'' at the American Institute of 
Mathematics in Palo Alto. It is a pleasure to thank AIM for their hospitality and support. JS gratefully acknowledges the kind hospitality of the
Department of Mathematics at Stanford University during a several months research visit in 2013. The authors would like to thank Olivier
Biquard, Sergey Cherkis, Laura Fredrickson, Nigel Hitchin, Andrew Neitzke and Richard Wentworth for useful discussions. Sergey Cherkis also read
the manuscript carefully and offered a number of helpful suggestions for improving it. 
%
%
%
\section{Preliminaries on gauge theory and Higgs bundles}
%
\subsection{Holomorphic vector bundles}
To fix notation, we briefly recall some classical facts about gauge theory and holomorphic vector bundles. 
Good general references are~\cite[Chapter I and VII]{ko87} or~\cite[Chapter III and Appendix]{wgp08}.

\bigskip

Let $X$ be a Riemann surface of genus $\gamma\geq2$ with canonical line bundle $K_X$, which we usually denote just
as $K$. We also fix a metric on $X$ in the designated conformal class. Consider a complex vector bundle $E\to X$ of 
rank $r=\rk(E)$ and degree $d=\deg(E)$, where by definition, $d$ is the degree of the complex line bundle 
$\det E=\Lambda^r E$. The {\em slope} or {\em normalized degree} of $E$ is
$$
\mu=\mu(E):=\deg(E)/\rk(E).
$$
Up to smooth isomorphism, the pair $(r,d)$ determines $E$ completely \cite[Chapter I.3]{po97}. We write 
$\GL(E)$ and $\SL(E)$ for the bundles of automorphisms, and automorphisms with determinant one, of $E$,
and set $\mf{gl}(E) = E\otimes E^*$, with $\mf{sl}(E)$ the subbundle of trace-free endomorphisms. The sections
$\Gamma(\GL(E))$ and $\Gamma(\SL(E))$ are the {\em complex gauge transformations} in this theory; these are 
infinite-dimensional Lie groups in the sense of~\cite{mi84}, with Lie algebras $\Omega^0(\mf{gl}(E))$
and $\Omega^0(\mf{sl}(E))$, respectively. A Hermitian metric $H$ on the fibres of $E$ determines the bundles $\U(E,H)$ 
and $\SU(E,H)$ of unitary and special unitary automorphisms of $(E,H)$; the corresponding Lie algebra bundles are 
$\mf u(E,H)$ and $\mf{su}(E,H)$. The sections $\Gamma(\U(E,H))$ are the {\em unitary} gauge transformations. 
For simplicity we usually omit mention of the metric $H$ in this notation. 

The affine space $\mc U(E)$ of unitary connections (with respect to $H$) has $\Omega^1(\mf u(E))$ as its group of translations. 
The action of the unitary gauge group $\U(E)$ on $\mc U(E)$ is the familiar one: 
\begin{equation}\label{gau.cov.der}
d_A\mapsto d_{A^g}:=g^{-1}\circ d_A\circ g=d_A+g^{-1}d_Ag.
\end{equation}
In the sequel, we tacitly fix a base connection $A_0$ and hence identify an arbitrary unitary connection $A$ with an element 
in $\Omega^1(\mf{u}(E))$. The covariant derivative $d_A:\Omega^0(E)\to\Omega^1(E)$ 
satisfies $d H(s_1,s_2)=H(d_As_1,s_2)+ H(s_1,d_As_2)$; in a local trivialization, $d_As=ds+As$, where $d$ is the usual differential 
and the {\em connection matrix} $A$ 
is a matrix-valued $1$-form. Under a local change of frame or {\em gauge} $g:U\to\GL(r)$, the connection matrix transforms as
$$
A\mapsto A^g:=g^{-1}Ag+g^{-1}dg
$$
which is consistent with~\cref{gau.cov.der}.
In a {\em Hermitian frame} $(s_1,\ldots,s_r)$, $A$ is $\mf u(r)$-valued.  These three perspectives, 
regarding $A$ as a point in $\mc U(E)$, a covariant derivative $d_A$ or as a connection matrix, are used 
interchangeably below.  In particular, $A=0$ can mean that $A$ is the base connection, that $d_A$ is given locally as $d$, 
or that the connection matrix vanishes. 

From the natural extension $d_A:\Omega^p(E)\to\Omega^{p+1}(E)$ we obtain the {\em curvature} of $A$, $F_A=d_A\circ d_A\in\Omega^2(\mf u(E))$,
which satisfies the familiar transformation rule
\[
F_{A^g}=g^{-1}F_A g.
\]
A unitary connection induces a unitary connection on any bundle derived from $(E,H)$, and in particular, this 
connection on $\det E$ is written $\det A$. By Chern-Weil theory, the degree of $E$ equals 
\[
d=\frac{i}{2\pi}\int_X\Tr F_A=\frac{i}{2\pi}\int_XF_{\det A}.
\]

We now explain the action of the complex gauge group on connections. 
An atlas of holomorphic trivializations of $E$ defines a {\em holomorphic structure} on $E$, and the Cauchy-Riemann operator 
$\delb$ acting on $\C^r$-valued functions in any local holomorphic chart yields a global {\em pseudo-connection} 
$\delb_E:\Omega^0(E)\to\Omega^{0,1}(E)$, where $\delb_E\circ\delb_E=0$. Conversely, any such pseudo-connection defines 
a holomorphic structure. Since $\delb_E^2 = 0$ holds trivially on a Riemann surface, any choice of pseudo-connection (which 
always exists) defines a holomorphic structure on $E$. The space of pseudo-connections $\mc C(E)$ is once again affine, and 
modelled on $\Omega^{0,1}(\mf{gl}(E))$. In a local holomorphic trivialization, $\delb_Es=\delb s+\alpha s$, so that 
$\delb_E=\delb$. We also write $\delb_\alpha$ for $\delb_E$ when we wish to emphasize the connection matrix. When there 
is no risk of confusion, we simply write $\delb$ for $\delb_E$ or $\delb_\alpha$. The complex gauge group 
$\Gamma(\GL(E))$ acts on $\mc C(E)$ by 
\[
\mc C(E)\to\mc C(E),\quad\delb_\alpha\mapsto\delb_{\alpha^g}:=g^{-1}\circ\delb_\alpha\circ g=\delb_\alpha+g^{-1}\delb_\alpha g.
\]
As before, the transformation rule for the connection matrix $\alpha$ under local gauge transformations is
\[
\alpha\mapsto\alpha^g=g^{-1}\alpha g+g^{-1}\delb g.
\]

If $A$ is a unitary connection (for some fixed Hermitian metric $H$), then the projection of $d_A$ onto $(0,1)$ forms, 
$$
\delb_A:=\mr{pr}^{0,1}\circ d_A, 
$$
is a pseudo-connection and hence determines a holomorphic structure; we also define $\partial_A=\mr{pr}^{1,0}\circ d_A$. 
Conversely, given the Hermitian metric $H$, then to any pseudo-connection $\delb_\alpha$ we can uniquely associate a
unitary connection $A=A(H,\delb_\alpha)$; this is the so-called {\em Chern connection}, which has $\delb_A=\delb_\alpha$. This correspondence is given by
\[
\mc C(E)\to\mc U(E,H),\quad\alpha\mapsto A(H,\delb_\alpha)
\]
where $\del_A=\del_{ A(H,\delb_\alpha)}$ is determined by the identity $\delb (H(s_1,s_2))=H(\delb_\alpha s_1,s_2)+H(s_1,\del_A s_2)$. In terms of 
local connections matrices, 
\[
A(H,\delb_\alpha)=\alpha-\alpha^*.
\]
The natural action of $\Gamma(\U(E,H))$ on $\mc U(E,H)$ thus extends to an action by elements of $\Gamma(\GL(E))$ by 
\[
A(H,\delb_\alpha)^g:=A(H,\delb_{\alpha^g})
\]
or equivalently, 
\begin{equation}
d_{A^g}=\delb_{A^g}+\del_{A^g}:=g^{-1}\circ\delb_A\circ g+g^*\circ\del_A\circ g^{*-1}.
\label{cplxgaugeaction}
\end{equation}
Note that this reduces to the action of~\eqref{gau.cov.der} when $g\in\Gamma(\U(E,H))$. The curvature transforms as
\be\label{cur.com.gt}
F_{A^g}=g^{-1}(F_A+\delb_A(G\cdot\partial_AG^{-1}))g
\ee
where $G=gg^*$.
%
\subsection{Hitchin's equations}
Fix a Hermitian vector bundle $(E,H)\to X$ of rank $r$ and degree $d$.  We shall be studying solutions $(A,\Phi)$ of Hitchin's 
self-duality equations~\cite{hi87} 
\be
\label{hit.equ.uni}
\begin{array}{rcl}
F_A+[\Phi\wedge\Phi^*] & = & -i\mu(E)\Id_E\omega, \\[0.4ex]
\delb_A\Phi & =& 0.
\end{array}
\ee
Here $A\in\mc U(E)$ and $\Phi \in \Omega^{1,0}(\mf{gl}(E))$ is called a {\em Higgs field}.

The unitary gauge group $\Gamma(\U(E))$ acts on Higgs fields by conjugation $\Phi^g:=g^{-1}\Phi g$ and it is not hard to see that it
therefore acts on the solution space of~\eqref{hit.equ.uni}.  Moreover, any solution $(A,\Phi)$ determines a {\em Higgs bundle} $(\delb,\Phi)$,
i.e.\ a holomorphic structure $\delb=\delb_A$ for which $\Phi$ is holomorphic: $\Phi\in H^0(X,\End(E)\otimes K)$. To do so we simply 
forget the first equation in \eqref{hit.equ.uni}.  Conversely, given a Higgs bundle $(\delb,\Phi)$, we ask whether $\delb$ can be extended 
to a unitary connection $A$ such that the first Hitchin equation holds. We say that a Higgs bundle $(\delb,\Phi)$ is {\em stable} if and 
only if $\mu(F)<\mu(E)$ for any nontrivial proper $\Phi$-invariant holomorphic subbundle $F$. (This $\Phi$-invariance means that $\Phi(F)\subset F\otimes K$.) 

\begin{example*}
The {\em determinant} of a Higgs field $\Phi$ is the holomorphic quadratic differential $\det\Phi\in H^0(X,K^2)$. Since any holomorphic section of $K^2$ has precisely 
$4(\gamma-1)$ zeroes (counted with multiplicity) and we are assuming that $\gamma > 1$, the set $\mf p_\Phi$ of zeroes of $\det\Phi$ is nonempty, and we write 
$X^\times_\Phi=X\setminus\mf p_\Phi \subsetneq X$ for its complement. (When there is no risk of confusion, we simply write $\mf p$ and $X^\times$.) 
A Higgs field is called {\em simple} if the zeroes of $\det\Phi$ are simple; in this case, 
$\mf p_\Phi$ has precisely $4(\gamma-1)$ zeroes, and by a standard local computation, if $p\in\mf p_\Phi$, 
then there exists a holomorphic coordinate chart centered at $p$ such that $\det \Phi=-z\,dz^2$. We always 
work with such a coordinate system near each $p$ and write $\Phi=\varphi \, dz$ so that $\det\Phi=\det\varphi\,dz^2$. For instance, the so-called
 fiducial Higgs field $\Phi_t^\fid$, $t < \infty$ which will be constructed in \cref{inn.fid.sol} is simple in this sense.

When the rank of $E$ is $2$ and $\Phi$ is a simple Higgs field, then necessarily the Higgs pair $(E,\Phi)$ is stable. 
Indeed, if there were to exist a holomorphic line bundle $L\subset E$ which is preserved by $\Phi$, then there are local 
holomorphic coordinates and frames such that
\[
\Phi=\varphi(z)dz=\begin{pmatrix}a(z)&b(z)\\0&c(z)\end{pmatrix}\, dz,
\]
where $a(z)$, $b(z)$ and $c(z)$ are holomorphic functions. Thus $\det\Phi(z)=a(z) c(z)$. Hence if this determinant vanishes simply at $z=0$, 
then either $a(0)=0$ or $c(0)=0$, but not both. On the other hand, $a(z)$ and $c(z)$ are the eigenvalues of the coefficient matrix $\varphi(z)$,
and by assumption, $\mathrm{tr}\, \varphi = 0$, i.e., $a(z) + c(z) = 0$, so if one of these terms vanishes then so must the other. We are grateful 
to Richard Wentworth for pointing out this simple but important fact.
\end{example*}

More generally,  $(E,\delb,\Phi)$ is called {\em polystable} if $(E,\delb,\Phi)=\oplus(E_j,\delb_j,\Phi_j)$ is a direct holomorphic sum of stable pairs 
$(E,\delb_j,\Phi_j)$ such that $\mu(E)=\mu(E_j)$ for all $j$. (Poly-)stability is clearly preserved by the action of the complex gauge group.

\medskip

{\bf Theorem A} (Hitchin, Simpson). {\em In the $\Gamma(\GL(E))$-orbit of the Higgs bundle $(\delb,\Phi)$, there exists a pseudo-connection 
which can be extended to a unitary connection $A$ solving~\eqref{hit.equ.uni} if and only if $(\delb,\Phi)$ is polystable. The connection 
$A$ is unique up to unitary gauge transformation.}

\medskip

This is due to Hitchin~\cite{hi87} in the rank $2$ case, and to Simpson~\cite{si88,si92} for higher rank. 

\begin{remark}
Theorem A is an existence theorem for a complex gauge transformation: if $A=A(H,\delb)$ is the Chern connection associated with the polystable 
Higgs bundle $(\delb,\Phi)$, then there exists (up to $\Gamma(\U(E))$) a unique $g\in\Gamma(\GL(E))$  such that $(A,\Phi)^g:=(A^g,\Phi^g)$ 
is a solution to~\eqref{hit.equ.uni}. 
\end{remark}

This result means that the moduli space
$$
\mc M_\GL(r,d)=\{(\delb,\Phi)\mid(\delb,\Phi)\mbox{ polystable}\}/\Gamma(\GL(E))
$$ 
of polystable bundles is identified with the space of solutions of \eqref{hit.equ.uni} modulo unitary gauge transformations:
$$
\mc M_\GL(r,d)\cong\{(A,\Phi)\mid\mbox{ solution of \eqref{hit.equ.uni}}\}/\Gamma(\U(E)).
$$

\medskip

{\bf Theorem B} (Hitchin, Nitsure, Simpson).
{\em The moduli space $\mc M_\GL(r,d)$ is a quasi-projective variety of (complex) dimension $2+r^2(2\gamma-2)$. It contains $\mc M^s_\GL$, 
the moduli space of stable Higgs bundles, as a smooth Zariski open set.}

\medskip

This again due to Hitchin~\cite{hi87} and Simpson~\cite{si88} in the rank $2$ and higher rank cases, respectively, and also
to Nitsure~\cite{ni91}, who proved it using GIT methods. 

\begin{remark}
If $\mr{gcd}(r,d)=1$, a polystable bundle is necessarily stable so that $\mc M_\GL(r,d)$ is a smooth, quasiprojective variety.
\end{remark}

Since $\Omega^1(\mf u(E))\cong\Omega^{0,1}(\mf{gl}(E)) :=\mc A$, the solution space of \eqref{hit.equ.uni} is a subspace of 
$\mc A\times\bar{\mc A}$. There is a natural $L^2$-Hermitian inner product
\be\label{l2.tra.pro}
\langle(\dot A,\dot\Phi),(\dot B,\dot\Psi)\rangle=2i\int_X\Tr(\dot A^*\wedge\dot B+\dot\Phi\wedge\dot\Psi^*),
\ee
and using this, $\mc A\times\bar{\mc A}$ carries a natural flat hyperk\"ahler metric. An infinite-dimensional variant of the 
hyperk\"ahler quotient construction~\cite{hklr87} yields

\medskip

{\bf Theorem C} (Hitchin).
{\em The space $\mc M^s_\GL(r,d)$ carries a natural hyperk\"ahler metric; this metric is complete 
when $\mr{gcd}(r,d)=1$.}

\medskip

In this paper we fix the determinant line bundle of $E$. According to the splitting $\mf{u}(r)=\mf{su}(r) \oplus \mf{u}(1)$, where 
$\mf{su}(r)$ is the set of trace-free elements of the Lie algbera $\mf{u}(r)$ and $\mf{u}(1) = i \R$, the bundle $\mf{u}(E)$
splits as $\mf{su}(E) \oplus i \underline{\R}$. If $A$ is a unitary connection, then its curvature $F_A$ decomposes as
$$
F_A = F_A^\perp + \frac{1}{r} \Tr(F_A) \otimes \Id_E,
$$
where $F_A^\perp \in \Omega^2(\mf{su}(E))$ is the {\em trace-free} part of the curvature and $\frac{1}{r} \Tr(F_A) \otimes \Id_E$ is 
the {\em pure trace} or {\em central} part, see e.g. \cite{po92}. 
Note that $\Tr(F_A) \in \Omega^2(i\underline{\R})$ is precisely the curvature of the induced connection on $\det E$. Let us fix a 
background connection $A_0$ from now on and consider only those connections $A$ which induce the same connection 
on $\det E$ as $A_0$ does, i.e.\ $A=A_0 + \alpha$ where $\alpha \in \Omega^1(\mf{su}(E))$;  in other words, any such
$A$ is trace-free ``relative'' to $A_0$. We may now consider the pair of equations 
\be\label{hit.equ.fixed.det}
\begin{array}{rcl}
F_A^\perp+[\Phi\wedge\Phi^*] & = & 0, \\[0.4ex]
\delb_A\Phi & =& 0,
\end{array}
\ee
for $A$ trace-free relative to $A_0$. Since the trace of a holomorphic Higgs field is constant, we may as well restrict to
trace-free Higgs fields $\Phi \in \Omega^{1,0}(\mf{sl}(E))$.  There always exists a unitary connection $A_0$ on $E$ 
such that $\Tr F_{A_0} =  -i \deg(E) \omega$, and with this as background connection, a solution of 
\eqref{hit.equ.fixed.det} provides a solution to \eqref{hit.equ.uni}, even though the latter system is a priori more stringent.

Define the moduli space
$$
\mc M^{\mathrm{gauge}}_\SL(r,d) := \{ (A_0+\alpha,\Phi)\mid\mbox{ solution of \eqref{hit.equ.fixed.det}}\}/\Gamma(\SU(E)).
$$
This does not depend in an essential way on the choice of the background connection $A_0$, we will choose $A_0$ 
as convenience dictates.

The choices above correspond to fixing a holomorphic structure $\bar\partial_{\det E}$ on $\det E$. We set
$$
\mc M_\SL(r,d):=\{ (\bar \partial, \Phi) \; \mbox{polystable}\mid\bar\partial \  \mbox{induces} \ \  \bar\partial_{\det E}, 
\Tr \Phi=0\}/ \Gamma(\SL(E)).
$$
The Kobayashi-Hitchin correspondence asserts that
$$
\mc M^{\mathrm{gauge}}_\SL(r,d) \cong \mc M_\SL(r,d). 
$$
The previous theorems carry over directly to the fixed determinant case, so in particular $\mc M_\SL(r,d)$ is a smooth 
quasiprojective variety of complex dimension $(r^2-1)(2\gamma-2)$, with a hyperk\"ahler metric which is complete provided $\mr{gcd}(r,d)=1$. 

\begin{remark}
If we were to consider the space of pairs $(A,\Phi)\in\mc A_0\times\bar{\mc A}_0$ which solve \eqref{hit.equ.fixed.det} modulo the 
gauge group of the principal $\P\U(r)$-bundle, then non-trivial isotropy groups necessarily occur, and hence the resulting moduli space 
is singular, cf.\ Hitchin's example~\cite[p.\ 87]{hi87}. It is therefore advantageous to work in the vector bundle setting.
\end{remark}

\bigskip

\noindent{\bf Conventions:}
For the rest of the paper, unless mentioned otherwise, we restrict attention solely to the fixed determinant case for complex vector 
bundles of rank $r = 2$, and with degree $d$ odd (so $\mr{gcd}(r,d) = 1$). We also write 
$$
\mc M:=\mc M_\SL(r,d),\quad\mc G^c:=\Gamma(\SL(E))\quad\mbox{and }\mc G:=\Gamma(\SU(E));
$$
these are the {\em moduli space of Higgs bundles}, and the {\em complex} and {\em unitary gauge groups}, respectively. These assumptions imply that $\mc M$ is a smooth quasiprojective variety of real dimension $12(\gamma-1)$ with a complete hyperk\"ahler metric.
%
%
%
\section{The fiducial solution}\label{inn.fid.sol}
Our first goal is to determine the model `fiducial' solutions of Hitchin's equations for Higgs fields with simple zeroes. These are the elements of a one-parameter radial family of `radial' global solutions on $\RR^2$, and are a key ingredient in the gluing construction below. The limiting element of this family is a pair $(A_\infty^\fid, \Phi_\infty^\fid)$ 
which is singular at $0$ and satisfies a decoupled version of Hitchin's equations: 
\begin{equation}
F_{A_\infty^\fid} = 0, \quad [ \Phi_\infty^\fid \wedge (\Phi_\infty^\fid)^*] = 0, \quad \mbox{and} \qquad 
\delb_{A^\fid_\infty} \Phi^\fid_\infty = 0.
\label{limhitch}
\end{equation}
The other elements of the family, $(A_t^\fid, t \Phi_t^\fid)$, $0 < t < \infty$, are smooth across $0$, satisfy \eqref{hit.equ.uni} (since
$E$ is trivial on $\CC$, $\mu(E) = 0$) and desingularize the limiting element. Further, they give rise to solutions of the self-dual Yang-Mills equation which are translation invariant in two directions and are also rotationally invariant. Symmetric solutions of this type (as well as others) have been intensively studied in connection with integrable systems, and Mason and Woodhouse show that the resulting reduced equation is essentially a Painlev\'e III \cite{mawo93}, see also \cref{sceqn}. On the other hand, some version of this family appears at least as far back as the paper of Ceccotti and Vafa \cite{ceva93}, but see also the more 
recent paper of Gaiotto, Moore and Neitzke \cite{gmn13}. Its existence can also be deduced from the work of Biquard and 
Boalch \cite{bibo04}, although their method does not give the explicit formula for it.  In any case, we present an 
explicit derivation of this family of solutions since this does not seem to appear in the literature. We are very grateful to 
Andy Neitzke for bringing this family of fiducial solutions to our attention and for explaining its main properties to us. We note that 
similar fiducial solutions in more general settings, e.g.\ Higgs fields with determinants having non-simple zeroes, or for higher 
rank groups, are being constructed in the forthcoming thesis of Laura Fredrickson~\cite{lafr} at UT Austin.

We begin with a useful lemma.
\begin{lemma}\label{norm.det}
Let $\Phi$ and $\Phi'$ be two Higgs fields on $X$ with $\det\Phi=\det\Phi'$ such that both $\Phi$ and $\Phi'$ are normal on $X^\times$.
Then there exists a unitary gauge transformation $g$ on $X^\times$ such that $\Phi^g =\Phi'$. 
\end{lemma}
\begin{proof}
Since $X^\times$ is homotopy equivalent to a bouquet of circles, any complex vector bundle over $X^\times$ is topologically trivial. 
More generally, any fibre bundle with connected fibre admits a global section over $X^\times$. In particular we may identify $\Phi$ and 
$\Phi'$ with functions $\varphi,\,\varphi':X^\times\to\mf{sl}(2,\C)$. Since $\varphi$ and $\varphi'$ are pointwise normal and have the same 
determinant, then locally on $X^\times$ we can find unitary gauge transformations $g$ such that $g^{-1}\varphi g=\varphi'$. Hence
\[
\mc{C}_{\varphi,\varphi'}=\{(p,g_p)\in X^\times\times\SU(2)\mid g_p^{-1}\varphi(p)g_p=\varphi'(p)\}\to X^\times
\]
is a smooth fibre bundle. The typical fiber is diffeomorphic to the pointwise stabilizer
\[
\operatorname{Stab}_{\, \SU(2)} \begin{pmatrix} \lambda & 0 \\ 0 & - \lambda \end{pmatrix} = 
\left\{ \begin{pmatrix} \tau & 0 \\ 0 & \bar \tau \end{pmatrix}\mid\tau \in S^1 \right\}
\]
which is a maximal torus $S^1 \subset \SU(2)$. Since this is connected, there exists a global section over $X^\times$.
\end{proof}
%
\subsection{The limiting fiducial connection}\label{lim.fid.con}
We first determine the {\em limiting fiducial solution} $(A_\infty^\fid, \Phi_\infty^\fid)$, where $A_\infty^\fid$ is flat and 
$\Phi_\infty^\fid$ is normal. In fact, we show that any pair $(A,\Phi)$ on $\CC$, where $A$ is a flat unitary connection with a simple
pole at $0$ and $\Phi$ is a normal Higgs field vanishing only at $0$ and with a simple zero there, can be modified by
a unitary gauge transformation to this particular model. 

The construction below can be carried out either on all of $\CC$ or else over an open disc $D$ centered at $0$. To be specific, 
we suppose the latter. As usual, $D^\times = D \setminus \{0\}$.  

Let $\Phi$ be normal. If $\Phi$ is a simple Higgs field on $D$, there is a complex coordinate $z$ such that $\det\Phi=-z\,dz^2$ on $D$. 
Fix a Hermitian metric $H$ on $E$ and corresponding unitary frame so that $E|_{D^\times}\cong D^\times\times\C^2$.
Define the {\em limiting fiducial Higgs field} with respect to this frame by 
\be\label{lim.fid.phi}
\Phi^\fid_\infty=\varphi^\fid_\infty \,dz:=\begin{pmatrix}0&\sqrt{|z|}\\\frac{z}{\sqrt{|z|}}&0\end{pmatrix}dz.
\ee
This is continuous on $D$ and smooth on $D^\times$. By \cref{norm.det}, since $\det\Phi^\fid_\infty=\det\Phi$, there is a unitary gauge 
transformation $g$ on $D^\times$, unique up to the unitary stabilizer of $\Phi^\fid_\infty$, which brings $\Phi$ into 
this fiducial form, that is, $g^{-1}\Phi g=\Phi^\fid_\infty$ over $D^\times$. The {\em infinitesimal complex stabilizer} of $\Phi^\fid_\infty$ 
is the bundle
$$
L_{\Phi^\fid_\infty}^\C:=\{\gamma\in\mf{sl}(E): [\gamma,\Phi^\fid_\infty]=0\}.
$$
In this fixed unitary frame, $\gamma\in\Omega^0(D^\times,L_{\Phi^\fid_\infty}^\C)$ if and only if
\be\label{gamma.mu}
\gamma_\mu=\mu\begin{pmatrix}0&1\\\frac{z}{|z|}&0\end{pmatrix},\quad\mu:D^\times\to\C.
\ee
Note that $\gamma_\mu$ is skew-Hermitian if and only if $e^{i\theta}\mu+\bar\mu=0$ (where $z = re^{i\theta}$); this reflects
the fact that this bundle of unitary stabilizers is a nontrivial $S^1$-bundle over $D^\times$ (cf.\ also the end of the proof of \cref{calc.ind.roots}).

\begin{proposition}\label{nor.form.fid}
Let $A$ be a flat unitary connection over $D^\times$ with respect to which $\Phi^\fid_\infty$ is holomorphic. Then there exists a 
unique unitary gauge transformation $g\in\Gamma(D^\times,\SU(E))$ stabilizing $\Phi^\fid_\infty$ and such that
\be\label{lim.fid.a}
A^g=A^\fid_\infty:=\frac{1}{8}\begin{pmatrix}1&0\\0&-1\end{pmatrix}\left(\frac{dz}{z}-\frac{d\bar z}{\bar z}\right).
\ee
Note that this limiting fiducial solution $(A^\fid_\infty,\Phi^\fid_\infty)$ is defined with respect to a fixed unitary {\em fiducial frame}. 
\end{proposition}
\begin{proof}
Write $A=A_r dr+A_\theta d\theta$ and name the components of these coefficient matrices with respect to the chosen fiducial frame as 
$$
A_r=\begin{pmatrix}i\beta&w\\-\bar w&-i\beta\end{pmatrix},\quad A_\theta=\begin{pmatrix}i\alpha&v\\-\bar v&-i\alpha\end{pmatrix}
$$
where $\alpha,\,\beta:D^\times\to\R$ and $v,\,w:D^\times\to\C$ are all smooth, and $z = re^{i\theta}$. 

We now show how the fact that $\Phi$ is holomorphic and $A$ is flat restricts these coefficients, and then 
use this information to gauge away the off-diagonal terms. 

\medskip

\noindent{\bf $\mb{\Phi}$ holomorphic:} 
We compute the terms in the equality 
\[
\delb_A\Phi^\fid_\infty :=\delb\Phi^\fid_\infty+[A^{0,1}\wedge\Phi^\fid_\infty]=0.
\]
First, 
\begin{equation}
\delb \Phi_\infty^\fid=\tfrac{1}{4}r^{-\frac 12}e^{i\theta}\begin{pmatrix}0&1\\-e^{i\theta}&0\end{pmatrix}d\bar z \wedge dz. 
\label{db1}
\end{equation}
Next, using $dr = \tfrac{1}{2}(e^{-i\theta}dz+e^{i\theta}d\bar z)$ and $d\theta = \tfrac{1}{2ir}(e^{-i\theta}dz-e^{i\theta}d\bar z)$, we have
\[
A^{0,1}=\tfrac{1}{2}e^{i\theta}(A_r+\tfrac{i}{r}A_\theta)d\bar z=\tfrac 12 e^{i\theta}\begin{pmatrix}-\frac{\alpha}{r}+i\beta&w+
\frac{i}{r}v\\ -\bar w -\frac{i}{r}\bar v&\frac{\alpha}{r}-i\beta\end{pmatrix}d\bar z,
\]
so that
\begin{multline}
[A^{0,1} \wedge \Phi_\infty^\fid]\\ = \frac12 r^{1/2} e^{i\theta} \begin{pmatrix}  
e^{i\theta} w + \bar{w} + \frac{i}{r} (e^{i\theta}v + \bar{v})&  2 (-\frac{\alpha}{r} + i\beta) \\
2 e^{i\theta} (\frac{\alpha}{r} - i\beta)  & - \left(e^{i\theta} w + \bar{w} + \frac{i}{r} (e^{i\theta}v + \bar{v})\right)
\end{pmatrix} 
d\bar z \wedge dz.
\label{db2}
\end{multline}
Adding \eqref{db1} to \eqref{db2} and equating coefficients to zero gives $\alpha=\frac 14$, $\beta=0$, and 
\be\label{a.v.w}
e^{i\theta}v+\bar v=e^{i\theta}w+\bar w=0.
\ee
We have used here the identity $e^{i\theta}u+\bar u=2e^{i\theta/2}\Re(e^{i\theta/2}u)$ (for any $u$) to separate into real and imaginary 
parts. Altogether, we have now obtained that 
\be\label{res.sec.step}
\begin{split}
&A=\begin{pmatrix}0&w\\-\bar w&0\end{pmatrix}dr+\begin{pmatrix} i/4 &v\\-\bar v&- i/4 \end{pmatrix}d\theta
\quad \mbox{and} \\
&A^{0,1}= \tfrac 12 e^{i\theta}\begin{pmatrix} -\frac{1}{4r}&w+\frac{i}{r}v\\-\bar{w} - \frac{i}{r}\bar v&\frac{1}{4r}\end{pmatrix}d\bar z
\end{split}
\ee
with $v,w$ subject to \eqref{a.v.w}.

\medskip

\noindent{\bf Flatness:} The equation $F_A = 0$ expands as 
$$
\partial_rA_\theta-\partial_\theta A_r+[A_r,A_\theta]=0.
$$
Substituting the expressions for $A_r$ and $A_\theta$ above now give that $\Im(\bar wv)=0$, which is in fact
the same as \eqref{a.v.w}, and more significantly, 
\be\label{a.is.flat}
\partial_rv = i P w, \qquad \mbox{where}\quad P = \tfrac{1}{i} \partial_\theta + \tfrac{1}{2}. 
\ee

\medskip

We now wish to find a gauge transformation $g_\mu$ in the stabilizer of $\Phi_\infty^\fid$ which simplifies $A$ even further. 
We assume that $g_\mu$ is the exponential of some section $\gamma_\mu$ of the infinitesimal stabilizer bundle, so using
the earlier expression for $\gamma_\mu$ we have that 
\begin{equation}
g_\mu = \begin{pmatrix} 
\cosh\big(e^{i\theta/2}\mu\big)&e^{-i\theta/2}\sinh\big(e^{i\theta/2}\mu\big)\\ 
e^{i\theta/2}\sinh\big(e^{i\theta/2}\mu\big)&\cosh\big(e^{i\theta/2}\mu\big)\end{pmatrix}
=:\begin{pmatrix}\eta_1 & \eta_2\\e^{i\theta}\eta_2 & \eta_1\end{pmatrix},\label{stab.gau.tra}
\end{equation}
where the final equality defines $\eta_1$ and $\eta_2$. Note that although $e^{\pm i\theta/2}$ is only defined on the 
slit domain $D^\times_- = D^\times \setminus (-1,0)$, both $\eta_1$ and $\eta_2$ make sense on all of $D^\times$. 

Now, $(A^{0,1})^{g_\mu}=g^{-1}_\mu A^{0,1}g_\mu+g^{-1}_\mu\delb g_\mu$, so we compute
\[
g^{-1}_\mu\delb g_\mu = 
\begin{pmatrix}
e^{2i\theta} \eta_2^2/4r &  e^{-i\theta} D\mu + \frac{1}{4r} e^{i\theta} \eta_1 \eta_2 \\[0.5ex]
D\mu - \frac{1}{4r} e^{2i\theta} \eta_1 \eta_2 & - e^{2i\theta} \eta_2^2/ 4r
\end{pmatrix} \, d\bar z,
\]
where we have written
\[
D = e^{i\theta/2} \del_{\bar{z}} e^{i\theta/2} 
\]
and are using the identity $\eta_1^2 - e^{i\theta} \eta_2^2 = 1$. Setting $U = w + (i/r) v$, and recalling from 
\eqref{a.v.w} that $\bar{w} + \frac{i}{r} \bar{v} = - e^{i\theta} U$, then further computation gives
\[
g^{-1}_\mu A^{0,1}g_\mu = \frac12 e^{i\theta}
\begin{pmatrix}
- (1/4r) (\eta_1^2 + e^{i\theta} \eta_2^2 ) & U -(1/2r) \eta_1 \eta_2 \\[0.5ex]
e^{i\theta}( (1/2r) \eta_1 \eta_2 + U)  &  (1/4r)(\eta_1^2 + e^{i\theta} \eta_2^2 )
\end{pmatrix} \, d\bar{z}.
\]
Adding these terms together yields 
\begin{equation}\label{traf.equ.a01}
(A^{0,1})^{g_\mu} = 
\begin{pmatrix}
-\frac{1}{8r} e^{i\theta}  & e^{-i\theta} D\mu + \frac12 e^{i\theta}U \\ 
D\mu + \frac12 e^{2i\theta} U & \frac{1}{8r} e^{i\theta}
\end{pmatrix}\, d\bar z.
\end{equation}

\medskip

Recall that our goal is to gauge away the off-diagonal components. To do this, we must choose $\mu$ so that
$D\mu + \frac12 e^{2i\theta} U = 0$. Using that
\[
D= e^{i\theta} \left( \del_{\bar{z}} - \frac{e^{i\theta}}{4r}\right), \quad \mbox{and}\quad 
\del_{\bar{z}} = \frac12 e^{i\theta} \left( \del_r + \frac{i}{r}\del_\theta\right),
\]
we write this equation, in terms of the operator $P$ in \eqref{a.is.flat}, as 
\begin{equation}
(\del_r - \frac{1}{r} P) \mu = - U := -w - \frac{i}{r} v.  
\label{tosolve}
\end{equation}

We solve this in a slightly unexpected way, by showing that the individual equations $\del_r \mu = -w$, $P\mu = iv$
are compatible. Indeed, $\del_r P \mu = P \del_r \mu$ is the same as $\del_r (iv) = P(-w)$, which follows precisely 
from the flatness of $A$ (as must be the case!). Noting that $P$ is invertible, we can now simply take $\mu = P^{-1} (iv)$,
and this satisfies both equations. 

The final point is that if we write $\overline{P}  = -Q$, where $Q = P-1$, then $Q(e^{i\theta}\mu) = e^{i\theta} P\mu$, so that
\[
Q( e^{i\theta} \mu + \overline{\mu}) = e^{i\theta} P\mu - \overline{P} \bar{\mu} = e^{i\theta} iv - \overline{ (iv)} = i( e^{i\theta}v + \bar{v}) = 0,
\]
by \eqref{a.v.w} again. Since $Q$ is also invertible, $e^{i\theta} \mu + \bar{\mu} = 0$, hence $\gamma_\mu$ is
skew-Hermitian and $g_\mu$ is a unitary gauge transformation, so that $A^g$ is still flat. 
\end{proof}
%
\subsection{The desingularized fiducial solutions}\label{des.fid.sol}
We now find a family of solutions $(A_t^\fid,\Phi_t^\fid)$ of Hitchin's rescaled equation
\begin{equation}\label{sca.hit.equ}
\mc H_t(A,\Phi)=(F_A+t^2[\Phi \wedge \Phi^*],\bar\partial_A\Phi),\quad t>0,
\end{equation}
which are smooth across $z = 0$ and which converge to $(A_\infty^\fid,\Phi_\infty^\fid)$ as $t\nearrow \infty$. Since this limiting pair is purely diagonal and purely off-diagonal, respectively, in the fiducial frame, it is natural to impose that $A_t^\fid$ and $\Phi_t^\fid$ have the same form. Thus we make the ansatz that in the same fiducial frame,
\be\label{tfid}
\begin{array}{rl}
A_t^\fid & = f_t(r) \begin{pmatrix} 1 & 0 \\ 0 & -1 \end{pmatrix} \left(\frac{dz}{z} - \frac{d\bar{z}}{\bar{z}}\right), \\
\Phi_t^\fid & =\varphi_t^\fid dz =\begin{pmatrix} 0 & r^{1/2} e^{h_t(r)} \\ r^{1/2}e^{i\theta} e^{-h_t(r)} & 0 \end{pmatrix}\, dz
\end{array}
\ee
(according to~\cite{mawo93} this ansatz essentially captures all possible solutions). We calculate that, 
\begin{multline*}
F_{A_t^\fid} + t^2[\varphi_t^\fid \wedge(\varphi_t^\fid)^*] \\ =
\left(\left( \frac{1}{\bar{z}} \delb_z f_t -\frac{1}{z}\delb_{\bar{z}}f_t\right) dz \wedge d\bar{z} +  
2rt^2\sinh(2h_t)\right)\, a_1 \\ = (\frac{1}{r} \del_r f_t -2rt^2\sinh(2h_t)) \, a_1, 
\end{multline*}
where $a_1=\begin{pmatrix} 1 & 0 \\ 0 & -1 \end{pmatrix}$, and in addition, 
$$
\delb_{A_t^\fid}\Phi_t^\fid = \left( \delb_{\bar{z}} \varphi_t^\fid  - \frac{f_t}{\bar{z}}[ a_1, \varphi_t^\fid ] \right) 
d\bar{z} \wedge dz = 0.
$$
After some computation, we are led to the pair of equations
\begin{eqnarray}
\del_r f_t(r) & = & 2t^2 r^2 \sinh 2h_t  \label{eqf} \\
f_t(r) & = &\frac{1}{8} + \frac{1}{4} r \partial_r h_t(r). \label{ffromh}
\end{eqnarray}
Now apply $r\partial_r$ to \eqref{ffromh} and insert into \eqref{eqf} to get
\begin{equation}
(r\partial_r)^2 h = 8  t^2 r^3 \sinh 2h. 
\label{painleve}
\end{equation}
To simplify this, set $\rho = \frac{8}{3} t r^{3/2}$, so that $r\partial_r = \frac{3}{2} \rho \partial_\rho$. 
Writing $h_t(r) = \psi(\rho)$ for some function $\psi$, we obtain 
\begin{equation}
(\rho \partial_\rho)^2 \psi = \frac{1}{2} \rho^2 \sinh 2\psi.
\label{sceqn}
\end{equation}
which is $t$-independent. Once we identify a suitable solution of this equation, we will have the solutions 
\begin{equation}
h_t(r) = \psi( \frac{8}{3}t r^{3/2}), \quad f_t(r) = \frac{1}{8} + \frac{1}{4} r \partial_r h_t
\label{sch}
\end{equation}
of the original system. The equation \eqref{sceqn} is of Painlev\'e type. It is known \cite{mtw77},~\cite{wi01} 
that there exists a unique solution which decays exponentially and has a the correct behavior as $\rho \to 0$, namely
\begin{equation}
\begin{array}{rl}
\bullet\ & \psi(\rho) \sim -\log (\rho^{1/3} \left( \sum_{j=0}^\infty a_j \rho^{4j/3}\right), \quad \rho \searrow 0 \\[0.5ex]
\bullet\ & \psi(\rho) \sim K_0(\rho) \sim \rho^{-1/2} e^{-\rho}, \quad \rho \nearrow \infty\\[0.5ex]
\bullet\ & \psi(\rho)\mbox{ is monotonically decreasing (and hence strictly positive)}.
\end{array}
\label{proph}
\end{equation}
The notation $A \sim B$ indicates a complete asymptotic expansion. In the first case, for example, for each $N \in \mathbb N$, 
\[
\left|\rho^{-1/3} e^{-\psi(\rho)} - \sum_{j=0}^N a_j \rho^{4j/3}\right| \leq C \rho^{4(N+1)/3},
\]
with a corresponding expansion for any derivative. The function $K_0(\rho)$ is the Macdonald function (or Bessel function 
of imaginary argument) of order $0$; it has a complete asymptotic expansion involving terms of the
form $e^{-\rho} \rho^{-1/2-j}$, $j \geq 0$, as $\rho \to \infty$. 

All of these calculations were sketched to us in a personal communication by Andy Neitzke, and we gratefully acknowledge his assistance. 

\medskip

From \eqref{proph} we can now compute the asymptotics of $f_t(r)$ and $h_t(r)$. 

\begin{lemma}\label{f_t-h_t-function}
The functions $f_t(r)$ and $h_t(r)$ have the following properties: 
\begin{enumerate}[1.]
\item[a)] As a function of  $r$, $f_t$ has a double zero at $r=0$ and increases monotonically from $f_t(0) = 0$ to the limiting 
value $1/8$ as $r \nearrow \infty$.  In particular, $0 \leq f_t \leq \frac 18$.

\item[b)]  As a function of $t$, $f_{t}$ is also monotone increasing. Further, $\lim_{t \nearrow \infty} f_t = f_\infty \equiv \frac18$ 
uniformly in $\mc C^\infty$ on any half-line $[r_0,\infty)$, for $r_0 > 0$. 

\item[c)] There are uniform estimates 
\[
\sup_{r >0}  r^{-1} f_t(r) \leq C t^{2/3} \quad \text{and}\quad \sup_{r >0} r^{-2} f_t(r) \leq C t^{4/3},
\]
where $C$ is independent of $t$. 

\item[d)] When $t$ is fixed and $r \searrow 0$, $h_t(r) \sim -\tfrac{1}{2} \log r + b_0 + \ldots$, where $b_0$ is an explicit constant.
On the other hand, $|h_t(r)| \leq C \exp( -\tfrac83 t r^{3/2})/ ( t r^{3/2})^{1/2}$ uniformly for $t \geq t_0 > 0$, $r \geq r_0 > 0$. 
\end{enumerate}
\end{lemma}
\begin{proof}
Define $\eta(\rho) = \tfrac 18 + \tfrac38 \rho\psi'(\rho)$, where $\rho=\frac{8t}{3}r^{3/2}$, so that $f_t(r)=\eta(\rho)$. 
By \eqref{sceqn}, 
\[
\eta'(\rho) = \tfrac 38 \rho\bigl(\psi''(\rho) + \rho^{-1}\psi'(\rho)\bigr) = \tfrac{3}{16}\rho \sinh(2\psi(\rho)),
\]
which implies that $\eta'(\rho) \geq 0$ since $\psi\geq0$.  In fact, \eqref{proph} also implies that $\lim_{\rho \to \infty}\eta(\rho)=\frac 18$
and 
\begin{equation}\label{eq:asymptf}
\eta(\rho)\sim\frac{1}{8}+\frac{3}{8}\rho\left(-\frac{1}{3\rho}-\frac{4a_1}{3a_0}\rho^{\frac{1}{3}}+O(\rho^{\frac{4}{3}})\right)=-\frac{a_1}{a_0}\rho^{\frac{4}{3}}+O(\rho^{\frac{7}{3}}),
\end{equation}
when $\rho$ is small, so $f_t$ has a double zero at $0$ as a function of $r$. This proves $a)$ and $b)$. 
Substituting $r=(\frac{3 \rho}{8t})^{2/3}$ now gives
\[
\frac{f_t}{r} = \left(\frac{8t}{3}\right)^{2/3}\frac{\eta(\rho)}{\rho^{2/3}} \quad \text{and} \quad \frac{f_t}{r^2} = \left(\frac{8t}{3}\right)^{4/3} 
\frac{\eta(\rho)}{\rho^{4/3}}.
\]
The estimates c) thus follow from \eqref{eq:asymptf}, which implies that $\eta(\rho)/\rho^{2/3}$ and $\eta(\rho)/\rho^{4/3}$ are 
bounded for $\rho > 0$. Finally, d) also follows directly from~\eqref{proph}.
\end{proof}

\begin{corollary}
The solutions $(A_t^\fid, \Phi_t^\fid)$ of the rescaled Hitchin equation are smooth at $z=0$. Further, they converge exponentially in $t$, uniformly in $\calC^\infty$ on any exterior region $r \geq r_0 > 0$ to $(A_\infty^\fid, \Phi_\infty^\fid)$.
\end{corollary}
\begin{proof}
The preceding Lemma gives that for fixed $t$, $r^{1/2} e^{h_t(r)} \sim c_0 + \ldots$ and $r^{1/2} e^{i\theta} e^{-h_t(r)} \sim z + \ldots$ 
as $r\to0$, and similarly, $f_t \sim c_1|z|^2 + \ldots$, while if $r$ is fixed, then 
\[
(A_t^\fid, \Phi_t^\fid) \longrightarrow (A_\infty^\fid, \Phi_\infty^\fid)
\]
exponentially in $t$, uniformly in $\calC^\infty$ on any exterior region $r \geq r_0 > 0$. 
\end{proof}
%
\subsection{The complex gauge orbit of the fiducial solutions}
We show now that all of the fiducial solutions $(A_t^\fid,\Phi_t^\fid)$ are equivalent under the complex gauge action. Towards that end, define in 
the fixed fiducial frame the pair
$$
A_0=0,\quad\Phi_0=\begin{pmatrix}0&1\\z&0\end{pmatrix}\, dz.
$$

\begin{proposition}\label{gau.orb.fid}
\begin{enumerate}[1.]
	\item Over $D$, the fiducial solution $(A_t^\fid,\Phi_t^\fid)$ is complex gauge equivalent to $(A_0,\Phi_0)$. In particular, all fiducial solutions for $0<t<\infty$ are mutually complex gauge equivalent. 
	\item Over $D^\times$, the limiting fiducial solution $(A_\infty^\fid,\Phi_\infty^\fid)$ is complex gauge equivalent to $(A_0,\Phi_0)$ by the singular gauge transformation
$$
g_\infty=\begin{pmatrix}|z|^{-\frac{1}{4}}&0\\0&|z|^{\frac{1}{4}}\end{pmatrix},
$$
i.e.,\ $(A_0,\Phi_0)^{g_\infty}=(A_\infty^\fid,\Phi_\infty^\fid)$.
\end{enumerate}
\end{proposition}

\begin{remark}
From~\eqref{proph} it follows that $A_t^\fid\to A_0$ as $t\to0$. However, $h_t(r)\sim-\log\tfrac 83 a_0\sqrt{r}t^{\tfrac 13}$ for small $t$ so that $\Phi^\fid_t$ actually diverges as $t\to0$.
\end{remark}

\begin{proof}
The second assertion is a straightforward computation so we focus on the first.  For simplicity, omit
the superscript `$\fid$' from all quantities. We seek a complex gauge transformation of the form
$$
g = \begin{pmatrix} e^{u_t} & 0 \\ 0 & e^{-u_t}\end{pmatrix}\in\Gamma(D,\SL(E)), \quad u_t = u_t(r), 
$$
so that $(A_0,\Phi_0)^g=(A^\fid_t,\Phi^\fid_t)$. Since $\partial_{\bar z} 
=\frac{1}{2}e^{i\theta}\partial_r$ on rotationally symmetric functions and $A_0 = 0$, we have 
$$
g^{-1}\circ\delb_{A_0}\circ g=\delb+g^{-1}\delb g=\delb+\frac{1}{2} e^{i\theta}
\begin{pmatrix}\partial_r u_t&0\\0&-\partial_r u_t\end{pmatrix}d \bar z.
$$
On the other hand,
$$
\delb_{A_t}=\delb-f_t\begin{pmatrix}1&0\\0&-1\end{pmatrix}\frac{d\bar z}{\bar z}.
$$
Thus $g^{-1}\circ\delb_{A_0}\circ g=\delb_{A_t}$ if and only if
\[
\partial_ru_t=-\frac{1}{4r}-\frac{1}{2}\partial_rh_t,
\]
which has the solution 
$$
u_t=-\frac{1}{4}\log r-\frac{1}{2}h_t.
$$
Hence $A_0^g=A_t$; moreover
$$
g^{-1}\Phi_0g=\begin{pmatrix}0&e^{-2u_t}\\ze^{2u_t}\end{pmatrix}dz,
$$
and $e^{-2u_t}=r^{\frac{1}{2}}e^{h_t}$, so that $g^{-1}\Phi_0g=\Phi_t$. 
\end{proof}
%
%
%
\section{Limiting configurations}\label{lim.conf}
We now start on the global aspects of this problem. As explained in the introduction, our existence theorem for solutions of \cref{hit.equ.uni} involves patching together copies of the fiducial solution with what we call a limiting configuration.  We have already explored these fiducial solutions, and our goal in this section is to describe the other building block, the limiting configurations.

\begin{definition}
Let $(\delb,\Phi)$ be a Higgs bundle, where $\Phi$ is simple, and suppose that $H$ is a Hermitian metric on the complex vector bundle $E$. 
A {\em limiting configuration} is a Higgs pair $(A_\infty,\Phi_\infty)$ over $X^\times$ which satisfies the decoupled Hitchin equations 
\begin{equation}
F_{A_\infty}^\perp=0,\quad[\Phi_\infty\wedge\Phi_\infty^*]=0,\quad\delb_{A_\infty}\Phi_\infty=0,
\label{dechiteq}
\end{equation}
and which agrees with $(A_\infty^\fid, \Phi_\infty^\fid)$ near each point of $\mf p_\Phi$, with respect to some holomorphic coordinate system and unitary frame for $E$.
Since we are in the fixed determinant case, we require $A_\infty$ and $\Phi_{\infty}$ to be trace-free, the former relative to some 
fixed background connection.
\label{lconfdef}
\end{definition}

The main objective in this section is to prove the following 

\begin{theorem}\label{ex.out.sol}
Let $(\delb,\Phi)$ be a Higgs bundle with simple Higgs field. Then there is a Hermitian metric $H_0$ so that if 
$A =A(H_0,\delb)$ is the associated Chern connection then the pair $(A,\Phi)$ is complex gauge equivalent via
some transformation $g_\infty\in\Gamma(X^\times,\SL(E))$ to a limiting configuration $(A_\infty,\Phi_\infty)$, i.e., 
$(A_\infty,\Phi_\infty):=(A,\Phi)^{g_\infty}$. 
\end{theorem}

\begin{remark}
As we will see below (\cref{sec:glue.const}, \cref{thm:nonlinearbvaluedisk}), every limiting configuration arises in this way.
\end{remark}

There are several steps in the proof. In the next subsection we describe a certain normal form for any simple Higgs field 
$\Phi$ on all of $X$. We then consider the problem of using some of the remaining gauge freedom (i.e., only those
gauge transformations which leave $\Phi$ in this normal form) to transform an initial connection to one with
vanishing trace-free curvature. This requires a brief foray into the theory of conic operators. After these steps
we are left with a limiting configuration in the sense of \cref{lconfdef}.  The final subsection considers the
local deformation theory of the space of limiting configurations. 
%
\subsection{Normal form for the Higgs field}
Fix the holomorphic bundle $(E,\delb)$ and let $\Phi$ be a simple Higgs field. We now show that $\Phi$ can be
brought to a simple normal form by a complex gauge transformation.  More specifically, we can smoothly 
``off-diagonalize'' $\Phi$ near each of its zeroes, and make it normal away from these zeroes.  Later in this section, 
we construct from $(E, \delb)$ (and an accompanying connection) a limiting configuration on all of $X$. 
Using Proposition~\ref{gau.orb.fid}, we can then patch in a smooth fiducial solution near each of the zeroes.
The resulting pairs $(A,\Phi)$ are then the first approximation to global solutions of Hitchin's equations.

The transformation of $\Phi$ near a simple zero to this normal form is elementary. 

\begin{lemma}\label{nor.for.zer}
In a neighbourhood of any simple zero of $\det\Phi$, there is a complex coordinate $z$ and a local complex frame 
of $E$ such that
\[
\Phi=\begin{pmatrix}0&1\\z&0\end{pmatrix}dz,\quad\det\Phi=-z\,dz^2.
\]
The frame can be chosen to be holomorphic if $\Phi$ is holomorphic.
\end{lemma}
\begin{proof}
Choose any complex frame for $E$ near some $p\in\mf p_\Phi$. Writing $\Phi = \varphi\, dz$ as usual, then since $p$ is a simple 
zero, $\varphi(0)$ must be nilpotent, 
but not the zero matrix (for if it were, then $\det \varphi$ would vanish like $z^2$). Applying a constant gauge transformation, 
we may thus assume that in some frame, 
\[
\varphi(z)=\begin{pmatrix}a(z) & b(z)\\c(z) &-a(z) \end{pmatrix}, \quad \mbox{with}\quad \varphi(0) = 
\begin{pmatrix} 0 & 1 \\ 0 & 0 \end{pmatrix}. 
\]
Since $\sqrt{b(z)}$ is well-defined and smooth near $0$, we can define the complex unimodular gauge transformation
\[
g(z)=\frac{1}{\sqrt{b(z)}}\begin{pmatrix}b(z) &0\\-a(z) &1\end{pmatrix},
\]
and then it is straightforward to check that $g^{-1} \Phi g$ takes the form in the statement of this lemma. 
\end{proof}

\begin{remark}
If the Higgs bundle $(\delb,\Phi)$ is described using spectral curves as in~\cite[Section 8]{hi87}, then~\cref{nor.for.zer} 
is also a direct consequence of the pushforward-pullback formula for vector bundles (see for instance~\cite[Chapter 2, Proposition 4.2]{hi99}).
\end{remark}

Before modifying $\Phi$ with a gauge transformation on the rest of the surface, let us choose a Hermitian metric $H_0$
which is particularly well adapted to this $\Phi$.  The important part of this definition is local near each zero $p_i\in\mf p_\Phi$. 
Thus choose a coordinate disc $(U_i,z_i)$ centered at $p_i$ and a holomorphic frame so that
$\Phi|_{U_i}$ equals the expression in \cref{nor.for.zer}. Define $H_0$ in $U_i$ by declaring this frame to be unitary. Now extend $H_0$ arbitrarily on the remaining part of $X$.
Associated to $H_0$ is its Chern connection $A$. The existence of a unitary holomorphic frame near each puncture implies that 
the connection matrix of $A$ in this frame vanishes. Finally, using \cref{gau.orb.fid} we can choose a complex gauge transformation 
$g\in\Gamma(\bigcup U_i^\times,\SL(E))$ such that $(A,\Phi)^g$ agrees with the fiducial solution. 

We now wish to extend this $g$ to the rest of $X$ so that $\Phi^g$ is normal outside the $U_i$. To motivate this, recall first that any invertible matrix $\varphi\in\mf{sl}(2,\C)$  may be conjugated (at a point) to be trace-free 
and diagonal. However, this diagonalization is impossible to do consistently on $X^\times$ because the eigenspaces are interchanged 
when traversing a loop surrounding any one of the $p_i$. We settle instead on the less ambitious goal of conjugating 
it to a normal matrix.

Define the subsets $D_\varphi$ and $N_\varphi$ of elements in $\SL(2,\C)$ which diagonalize and normalize $\varphi$,
respectively, at any point.  Fixing a basepoint $g_\varphi \in D_\varphi$, then 
$$
D_\varphi=\left\{g_\varphi\begin{pmatrix}\mu&0\\0&\mu^{-1}\end{pmatrix}\mid\mu\in\C^*\right\}\cup 
\left\{g_\varphi\begin{pmatrix}0&i\mu \\i\mu^{-1}&0\end{pmatrix}\mid\mu\in\C^*\right\} 
$$
and
$$
N_\varphi=D_\varphi\cdot\SU(2)=\left\{g_\varphi\begin{pmatrix}\mu&0\\0&\mu^{-1}\end{pmatrix}M\mid\mu\in\C^*,\,M\in\SU(2)\right\}.
$$

Because we have chosen the Hermitian metric $H_0$, we can speak about Hermitian adjoints and normal endomorphisms. Since 
any complex vector bundle is trivial over $X^\times$, we can write $\Phi=\varphi\otimes\kappa$ on this punctured surface, 
where $\kappa$ is a trivialization of $K$ over $X^\times$ and $\varphi \in \calC^\infty(X^\times;\mf{sl}(2,\C))$. 
There is a smooth fibration $\mc N_\varphi \to X^\times$, where each fibre $N_{\varphi(x)}$ is diffeomorphic to $\mc N:=N_{\Id}$.
If $g:U\to\SL(2,\C)$ diagonalizes $\varphi$ over $U$, then $\hat g(x,N)=g(x)N$
is a local trivialization of $\mc N_\varphi$ over $U$. Since the complex square root is well-defined over simply-connected sets, 
such a section $g$ always exists locally. However, the fibres $\mc N$ are homotopy-equivalent to $\SU(2)\cong S^3$,
while $X^\times$ retracts onto a bouquet of circles. There are thus no obstructions to extending sections. This proves the 

\begin{lemma}\label{ext.lem}
Any normalizing local section $g:U\to\mc N_\varphi$ on an open set $U \subset X^\times$ extends to a global section 
$X^\times\to\mc N_\varphi$. In particular, there exists a complex frame of $E|_{X^\times}$ with respect to 
which $\Phi$ is a normal matrix.
\end{lemma}
%
\subsection{Gauging away the trace-free part of the curvature}\label{spe.f.gau}
We can at last start the proof of~\cref{ex.out.sol}, and do so with a general observation. Given a Higgs pair $(A,\Phi)$ where $\Phi$ is simple,
~\cref{ext.lem} produces a field gauge-equivalent to $\Phi$ which is normal on $X^\times$, so we now 
assume that $\Phi$ is normal. This normalizing complex gauge transformation is not unique, however; 
we shall show how to use the remaining gauge freedom to transform $A$ to a projectively flat unitary 
connection, i.e., one for which $F_A^\perp = 0$. 

Recall now from \cref{lim.fid.con} that the infinitesimal complex stabilizer of $\Phi$ is a holomorphic line bundle $L_\Phi^\C=
\{\gamma\in\mf{sl}(E)\mid[\gamma,\Phi]=0\}$. Thus $L_\Phi:=L_\Phi^\C\cap\mf{su}(E)$ and $iL_\Phi$ are the skew-Hermitian 
and Hermitian elements. These are real line bundles over $X^\times$.  
The Jacobi identity shows that $L_\Phi$ is closed under the bracket $[\cdot\,,\cdot]=0$.

\begin{lemma}\label{cur.inf.sta}
If $\Phi$ is normal, and if $A$ is a unitary connection such that $\bar\partial_A\Phi=0$, then $F_A^\perp\in\Omega^2(L_\Phi)$.
\end{lemma}
\begin{proof}
This is a purely local statement. Choose a unitary eigenframe for $\Phi$, so in some local complex coordinate $z$, 
\[
\Phi=\left(\begin{array}{rr}\lambda&0\\0&-\lambda\end{array}\right)dz.
\]
The connection form $\alpha=\alpha^{0,1}-(\alpha^{0,1})^*$ is determined by its $(0,1)$-part 
\ben
\alpha^{0,1}=\left(\begin{array}{rr}a & b \\ c & d\end{array}\right)d\bar z.
\ee
Now 
\begin{multline*} 
\delb_A \Phi  =  \left(\left(\begin{array}{rr}\del_{\bar z}\lambda & 0 \\ 0 & -\del_{\bar z}\lambda\end{array}\right)+
\left[\left(\begin{array}{rr} a & b \\ c & d\end{array}\right),\left(\begin{array}{rr}\lambda & 0 
\\ 0 & -\lambda\end{array}\right)\right] \right)d\bar z\wedge dz  \\
 =  \left(\begin{array}{rr}\del_{\bar z}\lambda & -2b\lambda\\2c\lambda & -\del_{\bar z}\lambda\end{array}\right)d\bar z\wedge dz = 0
\end{multline*} 
implies $b=c=0$, so 
\be\label{con.mat.phi}
\alpha^{0,1}=\left(\begin{array}{rr}a&0\\0&d\end{array}\right)d\bar z.
\ee
In particular, $[\alpha\wedge\alpha]=0$, so $F_A=d\alpha$ and hence 
\[
F_A^\perp=\begin{pmatrix} \Re \partial_z (a-d) & 0 \\ 0 & \Re \partial_z(d-a)\end{pmatrix} dz\wedge d\bar z,
\]
as claimed.
\end{proof}

The bundles $L_\Phi$ and $iL_\Phi$ are parallel with respect to the induced unitary connection on $\mf{gl}(E)$. 
Indeed, $d_A\Phi=0$ (the $(1,0)$ part of the derivative automatically vanishes in this dimension), so 
$[d_A\gamma\wedge\Phi]=d_A[\gamma,\Phi]=0$. In particular, the connection Laplacian
$$
\Delta_A:=d_A^*d_A:\Omega^0(i\mf{su}(E))\to\Omega^0(i\mf{su}(E))
$$ 
restricts to a map $\Omega^0(iL_\Phi)\to\Omega^0(iL_\Phi)$.

\begin{proposition}\label{equ.s.gauge}
If $A$ is a unitary connection and $\gamma\in\Omega^0(iL_\Phi)$, then $F_{A^{\exp(\gamma)}}^\perp=0$ if and only if 
$\gamma$ is a solution to the Poisson equation
\be\label{poi.equ.gam}
\Delta_A\gamma=i\ast F_A^\perp.
\ee
\end{proposition}
\begin{proof}
By \cref{cur.com.gt}, if $g\in \Gamma(\SL(E))$, then 
$$
F_{A^g}^\perp = g^{-1}(F_A^\perp+\bar\partial_A(gg^*\partial_A(gg^*)^{-1})) g.
$$
Since $\gamma$ is Hermitian, $g=\exp(\gamma)=g^*$, and so $A^g$ is projectively flat provided that
\be\label{cur.gam.sta}
F_A^\perp+\delb_A\big(\exp(2\gamma)\partial_A\exp(-2\gamma)\big)=0.
\ee
Computing in a local unitary eigenframe for $\Phi$ then gives 
\begin{multline*}
\partial_A\exp(-2\gamma)=-2\exp(-2\gamma)\partial_A\gamma \\
\Longrightarrow 
\delb_A\big(\exp(2\gamma)\partial_A\exp(-2\gamma)\big)=-2\delb_A\partial_A\gamma.
\end{multline*}
Denote by $\Lambda$ the contraction with the K\"ahler form $\omega$ of $X$. Then, by~\cite[Prop.\ 1.4.21-22]{ni00},
\[
2i\Lambda\delb_A\partial_A\gamma=\Delta_A\gamma-2i\Lambda [F_A,\gamma]. 
\]
We use here the fact, which is straightforward to verify, that the induced connection $\End(A)$ on $\End(E)$
has curvature satisfying $F_{\End(A)}\gamma=[F_A,\gamma]$. However, by~\cref{cur.inf.sta}, $\Lambda [F_A,\gamma]= 
\ast [F_A,\gamma] = \ast [ F_A^\perp, \gamma] =0$, so~\eqref{cur.gam.sta} becomes $\Delta_A\gamma=i\ast F_A^\perp$.
\end{proof}
%
\subsection{Indicial roots}\label{sec:ind.roots}
At this point we have produced a Hermitian metric $H_0$ and a complex gauge transformation $g_0\in\Gamma(X^\times,
\SL(E))$ such that $(A,\Phi)^{g_0}$ consists of a normal Higgs field $\Phi^{g_0}$ and in an appropriate unitary
frame, $(A,\Phi)^{g_0}$ is fiducial near each $p_i\in\mf p_\Phi$. 

To simplify notation, let us replace $(A,\Phi)^{g_0}$ by $(A,\Phi)$ until further notice (near the end of this subsection). 
Because of the simple pole of $A$, the Poisson equation~\eqref{poi.equ.gam} is an example of an elliptic conic
operator, and we shall appeal to the theory of these operators to describe how to solve it. We refer to~\cite{mamo11} 
and the references therein for more on this theory. 
To be explicit, introduce polar coordinates in each punctured disk $U^\times$, and fix a trivialization of $iL_\Phi$ there
to identify sections with functions $\gamma:U^\times\to i\mf{su}(2)$. 
There is unitary frame in $U^\times$ so that 
\be\label{a.spe.fra}
A=\alpha d\theta=\frac{1}{4}\begin{pmatrix}i&0\\0&-i\end{pmatrix}d\theta.
\ee
The associated connection Laplacian is
$$
\Delta_A=\nabla^*_A\nabla_A=-\frac{1}{r^2}\left(\nabla_{r \partial_r}^2+\nabla_{\partial_\theta}^2\right).
$$
In the frame of~\eqref{a.spe.fra}, $\nabla_{r\partial_r}=r\partial_r$ and $\nabla_{\partial_\theta}=\partial_\theta+\alpha$, hence
$$
\Delta_A\gamma=-\frac{1}{r^2}\left((r\partial_r)^2\gamma+\partial_\theta^2\gamma+2[\alpha, \partial_\theta\gamma]+[\alpha,[\alpha,\gamma]]\right)=-(\partial_r^2+\frac{1}{r}\partial_r+\frac{1}{r^2}T)\gamma
$$
where $T$ is the $r$-independent {\em tangential operator}, acting on sections of the
restriction of $\mf{su}(E)$ over the $S^1$ link. 
The coefficients of $\Delta_A$ are smooth away from $\mf p_\Phi$, and are {\em polyhomogeneous} at these points. In other words,
near each such point, any coefficient $a$ has a complete asymptotic expansion 
\be\label{pol.hom.exp} 
a\sim\sum_{j}\sum_{k=0}^{N_j}r^{\nu_j}(\log r)^ka_{j,k}(\theta),
\ee
with a corresponding expansion for each of its derivatives. We encode the exponents which appear in this expansion as
an index set $\{\nu_j,N_j\}\subset\C\times\N$, which has the property that $\mbox{Re}\, \nu_j \to \infty$ as $j \to \infty$. 

\begin{definition}
A number $\nu\in\C$ is called an {\em indicial root} for $\Delta_A$ if there exists some $\zeta=\zeta(\theta)$ such 
that $\Delta_A (r^\nu\zeta(\theta)) = \calO(r^{\nu-1})$ (rather than the expected rate $\calO(r^{\nu-2})$). We let $\Gamma(\Delta_A)$ 
denote the set of indicial roots of $\Delta_A$. 
\end{definition}

Thus $\nu$ is an indicial root provided there is some leading order cancellation. It is not hard to see that $\nu \in \Gamma(\Delta_A)$ 
if and only if $-\nu^2$ is an eigenvalue for the tangential operator of $\Delta_A$ and $\zeta$ is the corresponding
eigenfunction, i.e., $(\nabla_{\partial_\theta}^2+\nu^2)\zeta(\theta)=0$.  \cref{con.op.main} below indicates the importance 
of this notion. Before turning to this, however, we compute the indicial roots for the connection Laplacian. 

\begin{lemma}\label{calc.ind.roots}
The set of indicial roots of $\Delta_A$ on sections of $i\mf{su}(E)$ is $\Gamma(\Delta_A)= \frac{1}{2}\Z$. 
On the other hand, $\Gamma(\Delta_{A}|_{iL_\Phi})=\frac{1}{2}+\Z$.
\end{lemma}
\begin{proof}
This is a local computation near each $p_i$, so we work in the fixed fiducial frame near any such point. 
Let $\{\tau_1,\,\tau_2,\,\tau_3\}$ be the standard basis of $\mf{su}(2)$, i.e.
\be\label{sta.bas.su2}
\tau_1=\begin{pmatrix}i&0\\0&-i\end{pmatrix},\quad \tau_2=\begin{pmatrix}0&1\\-1&0\end{pmatrix},\quad 
\tau_3=\begin{pmatrix}0&i\\i&0\end{pmatrix}.
\ee
Then $[\tau_1,\tau_2] = 2\tau_3$, $[\tau_2,\tau_3] = 2\tau_1$, $[\tau_3,\tau_1]=2\tau_2$ and 
the connection matrix $\alpha$ in~\eqref{a.spe.fra} equals $\tau_1/4$.  Thus writing 
\[
\zeta=i\zeta^1\tau_1+i\zeta^2\tau_2+i\zeta^3\tau_3,
\]
then 
\[
[\alpha,\partial_\theta\zeta]=\tfrac{1}{2}(-\partial_\theta\zeta^3i\tau_2+\partial_\theta\zeta^2i\tau_3), \qquad 
[\alpha, [\alpha, \zeta]]=-\tfrac{1}{4}(\zeta^2i\tau_2+\zeta^3i\tau_3),
\]
and hence 
\[
\nabla_{\partial_\theta}^2\begin{pmatrix}\zeta^1\\\zeta^2\\\zeta^3\end{pmatrix}=\begin{pmatrix}  \partial_\theta^2\zeta^1\\\partial_\theta^2\zeta^2-\partial_\theta\zeta^3-\frac {1}{4}\zeta^2\\
\partial_\theta^2 \zeta_3+\partial_\theta\zeta^2-\frac{1}{4} \zeta^3\end{pmatrix}.
\]
Thus $\nabla_{\partial_\theta}^2\zeta+\nu^2\zeta=0$ if and only if
\begin{equation}\label{indicial}
(\partial_\theta^2+\nu^2)\zeta^1=0,\quad \mbox{and} \quad
\begin{array}{rcl}
(\partial_\theta^2-\tfrac{1}{4}+\nu^2)\zeta^2-\partial_\theta\zeta^3 & =& 0, \\[0.3ex]
(\partial_\theta^2-\tfrac{1}{4}+\nu^2)\zeta^3+\partial_\theta\zeta^2 & = & 0.
\end{array}
\end{equation}
The first equation here is uncoupled, and its indicial roots are the integers.  On the other hand, 
restricting the coupled system to the span of $\zeta_\ell(\theta)=e^{i\ell\theta}/ \sqrt{2\pi}$, $\ell \in \Z$, 
then there is a homogeneous solution if and only if 
\[
\det \begin{pmatrix} -\ell^2-\frac{1}{4}+\nu^2&-i\ell\\ i\ell&-\ell^2-\frac{1}{4}+\nu^2 \end{pmatrix} =0,
\]
which occurs precisely when $\nu=\pm|\ell\pm 1/2|$.  Putting these two cases together shows that 
every $\ell/2$, $\ell \in \Z$, is an indicial root. 

Let us now compute the indicial roots for the restriction of $\Delta_A$ to sections of $iL_\Phi$. On $U$, where $\Phi$ is in 
fiducial form, $iL_\Phi$ is spanned by $\sigma(\theta)=\sin(\theta/2) i\tau_2+\cos(\theta/2)i\tau_3$ (which 
equals $-e^{-i\theta/2}\gamma_1$ in the notation of \cref{lim.fid.con}).
Write $\zeta(\theta)=f(\theta)\sigma(\theta)$ with $f(2\pi)=-f(0)$. Then $\nabla_{\partial_\theta}^2\zeta+\nu^2\zeta=0$ if 
and only if $\partial_\theta^2f +\nu^2f=0$. The space $\{f\in H^{2}(\R)\mid f(\theta+2\pi)=-f(\theta)\}$ is spanned by the 
functions $\{\zeta_{\ell+1/2}\}_{\ell\in\Z}$, so this equation has a nontrivial solution if and only if $\nu\in\Z+1/2$. 
\end{proof}

We finally turn to the problem of solvability of \eqref{poi.equ.gam}. To state the main result, let us first introduce 
appropriate function spaces. Let $\mc V_b$ denote the span over $\calC^\infty$ of the vector fields $r\partial_r$ 
and $\partial_\theta$. The corresponding $L^2$-based weighted $b$-Sobolev spaces are defined as follows. First, for $\ell \in \N$, set 
\[
H^\ell_b(\mf{su}(E))=\{u\in L^2(X)\mid V_1\ldots V_ju\in L^2(\mf{su}(E))\mbox{ for all }j\leq \ell,\,V_i\in\mc V_b\},
\]
and then define, for $\delta\in\R$, 
\[
r^\delta H^\ell_b(\mf{su}(E))=\{r^\delta u\mid u\in H^\ell_b(\mf{su}(E))\}.
\]
Since the area form is $r dr d\theta$, then locally near $r=0$, 
\[
r^\nu \in r^\delta H^\ell_b \Leftrightarrow \nu > \delta-1.
\]
This explains various index shifts below. We note, in particular, that
\[
-1/2 < \nu < 1/2 \Leftrightarrow 1/2 < \delta < 3/2. 
\]

From the basic definitions, 
\[
\Delta_A: r^\delta H^{\ell+2}_b(i L_\Phi)\to r^{\delta-2}H^\ell_b( iL_\Phi)
\]
is bounded for every $\delta$ and $\ell$.  The main result shows when this map is Fredholm.

\begin{proposition}\label{con.op.main}
Fix a real number $\nu \not\in\Gamma(\Delta_A|_{i L_\Phi})$ and define $\delta = \nu + 1$. 
\begin{itemize}
\item[i)] The operator 
\[
\Delta_A:r^\delta H^{\ell+2}_b(\mf{su}(E))\to r^{\delta-2}H^\ell_b(\mf{su}(E))
\]
is Fredholm, with index and nullspace remaining constant as $\delta$ varies over each connected component 
of $1 + (\R\setminus\Gamma(\Delta_A))$.
\item[ii)] Suppose that $\Delta_A\zeta=\eta \in r^{\delta-2}H^\ell_b(\mf{su}(E))$, where $\zeta\in r^\delta L^2(\mf{su}(E))$.
Then $\zeta\in  r^\delta H^{\ell+2}_b(\mf{su}(E))$. If $\eta$ is polyhomogeneous, then so is $\zeta$, and the exponents in
the expansion of $\zeta$ are determined by the exponents in the expansion for $\eta$ and the indicial roots $\nu_i 
\in \Gamma(\Delta_A)$ with $\nu_i>\delta-1$. In particular, any element of the nullspace of $\Delta_A$ is polyhomogeneous,
with terms in its expansion determined entirely by the indicial roots in this range. 
\end{itemize}
\end{proposition}

This is a straightforward adaptation of~\cite[Proposition 5 and 6]{mamo11}. The proof can be found in \cite{ma91}

The particular result needed for our immediate purposes is the

\begin{proposition}\label{sol.poiss.equ}
The mapping
\be\label{DA.fre.map}
\Delta_A: r^\delta H^{\ell+2}_b(iL_\Phi)\to r^{\delta-2}H^\ell_b(iL_\Phi)
\ee
is an isomorphism when $1/2<\delta < 3/2$. 
\end{proposition}
\begin{proof}
Since the interval $(-1/2,1/2)$ contains no indicial roots, \cref{con.op.main} shows that this map is Fredholm.  
The final statement of that result shows that any element of the nullspace of \eqref{DA.fre.map}, with $\delta$
in this range, is polyhomogeneous with leading term $r^{1/2}$.  We shall show below that this implies
that the nullspace is trivial. One further general remark is that the adjoint of \eqref{DA.fre.map} with weight 
$\delta$ can be identified with the corresponding map with weight $2-\delta$. Since the interval $(1/2, 3/2)$ 
is invariant under this reflection, it follows that the cokernel is also trivial, or in other words, 
\eqref{DA.fre.map} is an isomorphism as claimed. 

Thus it suffices to check that this mapping is injective, and we avail ourselves of the fact that 
if $\Delta_A \gamma = 0$ with $\varphi \in r^\delta L^2_b$, $1/2 < \delta < 3/2$, then $\gamma$
is polyhomogeneous with leading term $r^{1/2}$.  

Set $X^\times_\varepsilon = X^\times \setminus \bigcup B_\varepsilon(p_i)$. With $\gamma$ as above, we have
\[
0 = \int_{X^\times_\varepsilon }\langle\Delta_A\gamma,\gamma \rangle=\int_{X^\times_\varepsilon}|d_A \gamma|^2 + 
\int_{\del X^\times_\varepsilon}   \langle \del_\nu \gamma, \gamma \rangle.
\]
Since $\gamma \sim r^{1/2}$ and $\del_\nu \gamma \sim r^{-1/2}$, and the length of $\del X^\times_\varepsilon$ is of order
$\varepsilon$, the boundary term tends to zero. This proves that $\gamma$ is parallel with respect to $A$. 
However, since it vanishes as $r \to 0$, it must be identically $0$.  This proves the result.
\end{proof}

We apply this as follows. Let $A$ be the connection obtained at the end of the last subsection. Although it
has simple poles at the points of $\mf{m}_\Phi$, it is flat in a neighborhood of these points. This means
that the right hand side of \eqref{poi.equ.gam} vanishes near each $p_i$, hence the solution $\gamma$
of this equation is polyhomogeneous and vanishes like $r^{1/2}$ at these points.  We obtain, therefore,
a complex gauge transformation $g_1 = \exp \gamma$ such that $\Phi^{g_1}=\Phi$ and the trace-free part
of the curvature of $A^{g_1}$ vanishes. 

\bigskip

Resetting notation back to the initial Higgs pair $(A,\Phi)$, we have now produced a gauge-equivalent Higgs pair 
$(A,\Phi)^{g_0g_1}$ consisting of a projectively flat unitary connection $A$ and a normal Higgs field which is fiducial 
near the punctures in a certain unitary frame.  Note that $A^{g_0g_1}$ may not be in fiducial form, but 
applying \cref{nor.form.fid} gives a unitary gauge transformation $g_2\in\Gamma(\bigcup U_i^\times,\U(E))$ 
which stabilizes $\Phi^{g_0}$ and which can be extended to a global unitary gauge transformation over $X^\times$. 
Finally, $g_\infty=g_2g_1g_0\in \Gamma(X^\times,\SL(E))$ is the complex gauge transformation for which we have been
searching. This finishes the proof of \cref{ex.out.sol}.\hfill$\blacksquare$       
%
\subsection{Deformation theory of limiting configurations}\label{def.lim.config}
Fix a holomorphic quadratic differential $q$ and consider limiting configurations $(A_\infty,\Phi_\infty)$ with $\det \Phi_\infty = q$. 
We want to study the moduli space of these up to unitary gauge transformations. By \cref{norm.det} again, we see that if $\Phi_\infty$ and $\Phi_\infty'$ are Higgs fields with $\det\Phi_\infty = \det \Phi_\infty'$ and which are normal on $X^\times=X \setminus \mathfrak p_{\Phi_\infty}$, then there exists a gauge 
transformation $g \in \Gamma(X^\times,\SU(E))$ such that $g^{-1} \Phi_\infty g = \Phi_\infty'$. This leads us to study 
the solutions of 
$$
\bar\partial_A \Phi_\infty = 0, \qquad F_A^\perp = 0
$$
up to the action of the stabilizer of $\Phi_\infty$ in $\Gamma(X^\times,\SU(E))$. Writing $A = A_\infty + \alpha$, $\alpha \in \Omega^1(\mf{su}(E))$, this system is equivalent to
\[
[ \alpha \wedge \Phi_\infty] = 0, \qquad d_{A_\infty} \alpha = 0.
\]
Here we have used that $[\alpha \wedge \alpha]=0$ since $\alpha$ has values in the line bundle $L_{\Phi_\infty}$ in view of the first equation and the following lemma.
\begin{lemma}
For $\alpha \in \Omega^1(\su(E))$ and $\Phi \in \Omega^{1,0}(\sell(E))$ normal the following statements are equivalent:
\begin{enumerate}[i)]
\item $[\alpha \wedge \Phi] =0$;
\item $\alpha \in \Omega^1(L_{\Phi})$.
\end{enumerate}
\end{lemma}
\begin{proof}
Decompose $\alpha=\alpha^{1,0} + \alpha^{0,1}$. Then $[\alpha \wedge \Phi]=[\alpha^{0,1} \wedge \Phi]$ for dimensional reasons. Computing locally, i.e.\ writing $\alpha^{0,1} = \alpha_{\bar z} \, d\bar z$ and $\Phi = \varphi \, dz$, we get
\[
[\alpha^{0,1} \wedge \Phi] = [\varphi, \alpha_{\bar z}] \,dz \wedge d \bar z.
\] 
Assuming that $[\alpha \wedge \Phi]=0$ we therefore obtain $\alpha^{0,1} \in \Omega^{0,1}(L_\Phi^\C)$. Similarly, $\alpha^{1,0} \in \Omega^{1,0}(L_{\Phi^*}^\C)$. Now if $\Phi$ is normal, then $L_\Phi=L_{\Phi^*}$, such that $\alpha \in \Omega^1(L_\Phi)$. 
The converse is trivial.
\end{proof}

The determination of the infinitesimal deformation space amounts to a cohomology computation:

\begin{lemma}
If all zeroes of $q$ are simple, then 
\[
\dim_\R H^1(X^\times; L_{\Phi_\infty}) = 6\gamma-6,
\]
where $\gamma$ is the genus of $X$.
\end{lemma}
\begin{proof}
Since $L_\infty$ is a real line bundle, 
\[
\chi(X^\times;L_{\Phi\infty}) = \chi(X^\times) = 2-2\gamma-k
\]
where $k=|\mathfrak p|$ is the number of zeroes. There are no parallel sections since $L_{\Phi_\infty}$ is twisted near 
each $p_i$, i.e., $H^0(X^\times;L_{\Phi_\infty})=0$. With $M = X \setminus B_\varepsilon(\mathfrak p)$ (so
$\del M$ is a union of $k$ circles), Poincar\'e duality yields
\[
H^2( X^\times; L_{\Phi_\infty}) = H^2(M;L_{\Phi_\infty}) = H^0(M,\partial M;L_{\Phi_\infty})=0.
\]
Therefore 
\[
\dim_\R H^1(X^\times; L_{\Phi_\infty}) =k + 2\gamma-2 = 4\gamma-4 + 2\gamma-2 = 6\gamma-6
\]
as claimed.
\end{proof}

We see finally that in the long exact cohomology sequence for the pair $(M,\del M)$, the natural map
\[
H^1(M,\partial M; L_{\Phi_\infty}) \longrightarrow H^1(M; L_{\Phi_\infty}) 
\]
must be an isomorphism. 

\begin{corollary}\label{bound.stratum}
The moduli space of limiting configurations with determinant equal to a fixed holomorphic quadratic differential $q$ with simple 
zeroes is a torus of dimension $6\gamma-6$.
\end{corollary}
\begin{proof}
The action of $g \in \Stab_{\Phi_\infty}$ on a connection $A$ is given by
\[
g^{-1} \circ d_A \circ g = d_A + g^{-1}(d_A g) = d_A + d_A \log g
\]
where $g$ is a section of a nontrivial circle bundle (and $\log g$ a multivalued section of $L_{\Phi_\infty}$). Therefore the moduli 
space under consideration is simply the quotient of the de Rham cohomology space $H^1(X^\times;L_{\Phi_\infty})$ 
by the lattice of classes with integer periods. The result thus follows from the previous lemma. 
\end{proof}

\begin{remark}
This is consistent with \cite[Theorem 8.1]{hi87}, where it is shown that the space of Higgs bundles $(\delb,\Phi)$ with fixed determinant and
with simple zeroes is a $(3\gamma-3)$-dimensional Prym variety (and thus a $(6\gamma-6)$-dimensional real torus). 
\end{remark}
%
%
%
\section{The linearized problem}
%
\subsection{Linearization of the Hitchin operator}
For any Hermitian vector bundle $V\to X$ with connection $\nabla$, denote by $W^{k,p}(V)$ the usual
Sobolev space of sections $s$ with $\nabla^j s \in L^p$, $j \leq k$; we adopt the usual shorthand, 
writing $H^k(V)$ when $p=2$, etc. More generally, we also consider $W^{k,p}$ sections of fibre bundles. 

Since we are in the fixed determinant case, we fix a background connection $A_0$ now and consider the Hitchin operator 
\[
\mc H_t(A,\Phi) = (F^\perp_A+t^2[\Phi \wedge \Phi^*],\bar\partial_A\Phi)
\]
for connections $A$ which are trace-less relative to $A_0$ and trace-less Higgs fields $\Phi$. We further consider the orbit map
\begin{equation}
\mathcal{O}_{(A,\Phi)}(\gamma)=(A,\Phi)^g=(A^g,\Phi^g), \qquad g = \exp (\gamma). 
\label{orb.map}
\end{equation}
Our ultimate goal is to find a point in the complex gauge orbit of a given Higgs pair $(A,\Phi)$ which is in
the nullspace of $\mc H_t = 0$. Since the condition that $\bar{\partial}_A \Phi = 0$ is preserved under 
the complex gauge group, we in fact only need to find a solution of 
\begin{equation}
F_t(\gamma) := \pr_1 \circ \mathcal{H}_t \circ \mathcal{O}_{(A,\Phi)}(\exp (\gamma)) = 0. 
\label{eq:defmapFt}
\end{equation}
More explicitly, we wish to solve
\[
F_{A^g}^\perp+t^2[\Phi^g\wedge(\Phi^g)^*] = 0, \qquad g = \exp(\gamma).
\]

Using the continuity of the multiplication maps $H^1\cdot H^1\to L^2$ and $H^2\cdot H^1\to H^1$, it 
is straightforward that the three maps
\begin{align}
&\mc H_t \colon H^1(\Lambda^1\otimes \mf{su}(E)\oplus\Lambda^{1,0}\otimes \mf{sl}(E)) \to L^2(\Lambda^2\otimes\mf{su}(E)
\oplus\Lambda^{1,1}\otimes\mf{sl}(E)),\nonumber\\[0.5ex]
&\mathcal{O}_{(A,\Phi)}\colon H^2(i\mf{su}(E)) \to  H^1(\Lambda^1\otimes\mf{su}(E)\oplus\Lambda^{1,0}\otimes\mf{sl}(E)), \label{eq:nonlinmaps}\\[0.5ex]
&F_t\colon H^2(i\mf{su}(E)) \to L^2(\Lambda^2\otimes\mf{su}(E)),\nonumber
\end{align}
are all well-defined and smooth. 

We now compute the linearizations of these mappings.  First, the differential at $g = \mbox{Id}$ of \eqref{orb.map} is 
\[
\Lambda_{(A,\Phi)}\gamma = (\Lambda_A(\gamma), \Lambda_\Phi(\gamma))  = (\bar\partial_A\gamma-\partial_A\gamma^{\ast},[\Phi,\gamma]),
\]
so when $\gamma\in \Omega^0(i\mf{su}(E))$,
\[
\Lambda_{(A,\Phi)}\gamma=(\bar\partial_A\gamma-\partial_A\gamma,[\Phi,\gamma]).
\] 
Next, 
\[
D\mc H_t \begin{pmatrix} \dot A \\ \dot \Phi \end{pmatrix}= \begin{pmatrix}
 d_A & t^2( [\Phi \wedge {\cdot\,}^{\ast}]+[\Phi^{\ast}\wedge\cdot\,])\\[0.5ex]
[\Phi\wedge\cdot\,] & \bar{\partial}_A
 \end{pmatrix}\begin{pmatrix} \dot A \\ \dot \Phi \end{pmatrix}
\]
whence
\[
(D\mc H_t \circ \Lambda_{(A,\Phi)})(\gamma) = \begin{pmatrix}(\partial_A\bar\partial_A-\bar\partial_A\partial_A)
\gamma+t^2([\Phi\wedge[\Phi,\gamma]^{\ast}]+[\Phi^{\ast}\wedge[\Phi,\gamma]])\\[0.5ex] 
[\Phi\wedge(\bar\partial_A\gamma-\partial_A\gamma)]+\bar\partial_A[\Phi,\gamma]
\end{pmatrix}.
\]
The first component is precisely $DF_t(\gamma)$.  Using that $\bar\partial_A\Phi=0$, as well as the fact that
$[\Phi\wedge\partial_A\gamma]=0$ for dimensional reasons, the entire second component vanishes. Now recall 
from~\cite[Prop.\ 1.4.21 and 1.4.22]{ni00} the identities
\begin{eqnarray*}
2\bar\partial_A\partial_A=F_A-i\ast\Delta_A,\qquad2\partial_A\bar\partial_A=F_A+i\ast\Delta_A,
\end{eqnarray*}
as well as 
\[
[ \Phi \wedge [\Phi,\gamma]^*] = - [ \Phi \wedge [ \Phi^*,\gamma]],
\]
to rewrite 
\begin{equation}\label{eq:operatorDt}
DF_t(\gamma) =  i\ast\Delta_A\gamma+t^2M_\Phi \gamma,
\end{equation}
where
\[
M_\Phi\gamma:=[\Phi^{\ast}\wedge[\Phi,\gamma]] - [ \Phi \wedge [ \Phi^*,\gamma]].
\]  
Applying $-i \,\ast : \Omega^2(\su(E)) \to \Omega^0(i\su(E))$ finally yields the operator
\[
L_t(\gamma) = \Delta_A \gamma  - i \ast t^2 M_\Phi \gamma.
\]
Observe that
\begin{multline*}
\Lambda_{(A,\Phi)}\colon\Omega^0(\mf{sl}(E))\to\Omega^1(\mf{su}(E))\oplus\Omega^{1,0}(\mf{sl}(E)) \\
D\mc H_t \circ\Lambda_{(A,\Phi)}\colon\Omega^0(i\mf{su}(E))\to\Omega^2(\mf{su}(E))\oplus\Omega^{1,1}(\mf{sl}(E)) \\
\mbox{and} \qquad
L_t \colon\Omega^0(i \su(E)) \rightarrow \Omega^0(i \su(E)), 
\end{multline*}
are all bounded from $H^1$ to $L^2$, or $H^2$ to $L^2$ respectively.

Remarkably, $L_t \geq 0$: 

\begin{proposition}\label{positivity}
If $\gamma \in \Omega^0(i \su(E))$, then
\[
\langle \ast L_t \gamma, \gamma \rangle_{L^2} = t^{-2}\|d_A \gamma\|_{L^2}^2 + 4 \| [ \Phi,\gamma ]\|_{L^2}^2 \geq 0.
\]
In particular, $L_t \gamma =0$ if and only if $d_A\gamma=[\Phi,\gamma ] = 0$.
\end{proposition}

\noindent This follows directly from the 

\begin{lemma}
For $\gamma \in \Omega^0(i \su(E))$, 
\[
\langle-i\ast M_\Phi \gamma, \gamma \rangle = 4 | [\Phi,\gamma] |^2  \geq 0.
\]
In particular, $M_\Phi \gamma =0$ if and only if $[\Phi,\gamma]=0$.
\end{lemma}
\begin{proof}
Fix a local holomorphic coordinate $z$ so that $\Phi = \varphi\, dz$, hence $\Phi^* = \varphi^* d\bar z$. Then
\begin{multline*}
[\Phi^* \wedge [ \Phi,\gamma ]] = - [\varphi^*, [\varphi, \gamma ]] \,dz \wedge d \bar z,  \\\mbox{and} \quad
-[\Phi \wedge [ \Phi^*,\gamma ]] = - [\varphi, [\varphi^*, \gamma ]] \, dz \wedge d \bar z,
\end{multline*}
so that 
\[
M_\Phi \gamma = -([\varphi^*, [\varphi, \gamma ]] + [\varphi, [\varphi^*, \gamma ]]) \, dz \wedge d \bar z.
\]
We use the Hermitian inner product $\langle A, B \rangle = \Tr AB^*$ on $\mf{sl}(2,\C)$. Its $\ad$-invariance 
yields that $\langle [H,A],B \rangle = \langle A, [H^*,B] \rangle$ whenever $A, B, H\in\mf{sl}(2,\C)$. Therefore
\[
\langle [ \varphi^*, [ \varphi, \gamma ]], \gamma  \rangle =  | [ \varphi, \gamma]|^2 \quad \mbox{and}\quad
\langle [ \varphi, [ \varphi^*, \gamma ]], \gamma  \rangle =  | [ \varphi^*, \gamma]|^2 =  | [ \varphi, \gamma]|^2,
\]
and since $2i\ast 1=-dz \wedge d \bar z$, we deduce that
\[
\langle M_\Phi \gamma, i \ast \gamma \rangle = |[\varphi, \gamma]|^2  |dz \wedge d\bar z|^2 = 4  |[\varphi, \gamma]|^2,
\]
as claimed.
\end{proof}

In parallel with this discussion, fix $\varphi\in\mf{sl}(2,\C)$ and consider the operator
\begin{equation*}
M_\varphi: i\su(2) \to i\su(2),\quad\gamma \mapsto 2( [ \varphi^*, [ \varphi, \gamma]] + [\varphi, [ \varphi^*, \gamma]]).
\end{equation*}
Calculating as above, 
\begin{equation}\label{emmvarvieh}
\langle M_\varphi \gamma, \gamma \rangle = 2 |[\varphi,\gamma]|^2 + 2 |[\varphi^*,\gamma]|^2 = 4 |[\varphi,\gamma]|^2.
\end{equation}
Clearly $M_\varphi$ is Hermitian with respect to $\langle \cdot \,, \cdot \rangle$ and satisfies 
$g^{-1}(M_\varphi\gamma)g=M_{g^{-1}\varphi g}g^{-1}\gamma g$ when $g\in \SU(2)$.

\begin{lemma}
If $\varphi \in \mf{sl}(2,\C)$, then $M_\varphi : i \su_2 \to i \su_2$ is invertible if and only if $[\varphi, \varphi^*] \neq 0$. 
If $[\varphi,\varphi^*]=0$ for some $0 \neq \varphi \in \mf{sl}(2,\C)$, then $M_\varphi$ has a one-dimensional kernel.
\end{lemma}
\begin{proof}
Assume first that $\ker M_\varphi \neq \{0\}$. According to \cref{emmvarvieh}, there exists $\gamma \in i \su(2)$, $\gamma \neq 0$,
such that $[\varphi,\gamma]=0$. Since $\gamma$ has two distinct eigenvalues, there must exist a unitary basis in 
terms of which both $\gamma$ and $\varphi$ are diagonal. In particular, $\varphi$ is normal, i.e.\ $[\varphi,\varphi^*]=0$. 
Conversely, if $\varphi$ is normal, then $\ker M_\varphi=\{ \gamma \in i\su(2) : [\varphi,\gamma]=0\}$ is non-trivial, and 
this kernel is one-dimensional when $\varphi \neq 0$.
\end{proof}

Now take $\varphi$ to be the fiducial Higgs field,
\[
\varphi = \varphi_t^\fid=\begin{pmatrix} 0 & |z|^{\frac 12}e^{h_t(|z|)} \\ |z|^{\frac 12}e^{i\theta} e^{-h_t(|z|)} & 0 \end{pmatrix}.
\]

\begin{lemma}\label{lem:MtL2bound}
There is a uniform bound
\[
\sup_{z \in D_1(0)} |\varphi^\fid_t(z)| \leq C
\]
for some constant $C>0$.
\end{lemma}
\begin{proof}
As in \cref{des.fid.sol}, substitute $|z|^{\frac 12}=(\frac{3}{8}t^{-1}\rho)^{\frac 13}$. Uniform boundedness of the upper 
right entry $|z|^{\frac{1}{2}}e^{h_t}(|z|)$, $0\leq|z|\leq1$, is equivalent to uniform boundedness of the function
\begin{eqnarray*}
\rho\mapsto\big(\frac{3}{8}t^{-1}\rho\big)^{\frac{1}{3}}e^{\psi(\rho)}, \qquad 0\leq\rho\leq\frac{8t}{3}, 
\end{eqnarray*}
where $\psi$ is the function appearing in \eqref{sceqn}. 
Since $\psi$ decays exponentially as $\rho\to\infty$, it suffices to show that this map is also bounded for $\rho\to0$.
This follows easily from the asymptotic expansion \eqref{proph}. Uniform boundedness of the lower left entry amounts 
to boundedness of the function
\begin{eqnarray*}
\rho\mapsto\big(\frac{3}{8}t^{-1}\rho\big)^{\frac{1}{3}}e^{-\psi(\rho)}
\end{eqnarray*}
on the same interval, which can be proved as above. 
\end{proof}
Finally, we state the 

\begin{corollary}\label{emmviehtee}
There is a constant $C > 0$ such that 
\[
\sup_{z \in D_1(0)} |M_{\varphi^\fid_t(z)}| \leq C.
\]
\end{corollary}
%
\subsection{Local analysis of the linearization at a fiducial solution}\label{lin.op.fid}
In this section we analyze the linear operator $L_t$ on the disk $D = D_1(0)$, computed relative to a fiducial pair $(A_t^\fid,\Phi_t^\fid)$,
with the goal of determining sharp bounds for the norm of its inverse $G_t$. In what follows, we often omit the bundles from the 
function spaces. We also replace the $H^2$ norm with the equivalent graph norm for the standard Laplacian 
$\Delta=-((r\partial_r)^2+\partial^2_\theta)/r^2$, i.e.
\[
\|u \|_\Delta^2=\|u \|^2_{L^2}+\|\Delta u\|^2_{L^2}.
\]
We consider both $\Delta$ and 
\[
L_t:=\Delta_{A_t^\fid}+t^2M_{\Phi_t^\fid}
\]
with Dirichlet boundary conditions, or equivalently,
on the common domain $H^2(D) \cap H^1_0(D)$. Because of the nonnegativity of $t^2 M_{\Phi_t^\fid}$ and the positivity of the leading part, 
it is clear that
\[
L_t\colon H^2(D)\cap H^1_0(D)\to L^2(D)
\]
is injective, and since it is also self-adjoint, it is an isomorphism. Thus it has an inverse 
\[
G_t:=L_t^{-1}\colon L^2(D)\to H^2(D)\cap H^1_0(D).
\]
We are interested in understanding the norm of this inverse as $t \nearrow \infty$. We do this by reducing
$L_t$ to a family of ordinary differential operators. 

Trivialize the bundle $i\mf{su}(E)$ by the constant sections $\{\sigma_1=i\tau_1,\sigma_2=i\tau_2,\sigma_3=i\tau_3\}$,
cf.~\cref{sta.bas.su2}, so $[\tau_1,\sigma_1]=0$, $[\tau_1,\sigma_2]=2\sigma_3$ and $[\tau_1,\sigma_3]=-2\sigma_2$. Now 
consider the decomposition
\[
i\mf{su}(E)=\langle\sigma_1\rangle\oplus\langle\sigma_2,\sigma_3\rangle=:iV\oplus iV^\perp,
\]
where $iV = \mathrm{span}\,\{\sigma_1\}$ and orthogonality is with respect to $\langle A,B\rangle=\operatorname{tr}(AB)$ 
on $i\mf{su}(2)$. This splitting is parallel for the connection $A_t^\fid=2f_t\tau_1d\theta$. The restriction of
$\Delta_{A_t^\fid}$ to $iV$ is the scalar Laplacian, whereas 
\[
\left. \Delta_{A_t^\fid} \right|_{iV^\perp} =-\frac{1}{r^2}\left((r\partial_r)^2+\partial_\theta^2+
\begin{pmatrix}
-16f_t^2&-8f_t\partial_\theta\\8f_t\partial_\theta&-16f_t^2
\end{pmatrix}\right)
\]
acting on pairs $(a_2,a_3)^\top=a_2\sigma_2+a_3\sigma_3$. Conjugating by $M=\begin{pmatrix}1&1\\i&-i\end{pmatrix}$ 
provides a decoupling: 
\begin{align*}
M^{-1}\circ\Delta_{A_t^\fid}|_{V^\perp}\circ M&=-\frac{1}{r^2}\left((r\partial_r)^2+\partial^2_\theta+
\begin{pmatrix}
-8if_t\partial_\theta-16f_t^2&0\\0&8if_t\partial_\theta-16f_t^2
\end{pmatrix}
\right)\\
&= -\frac{1}{r^2} \left(  (r \partial_r)^2 + \begin{pmatrix}
 (\partial_\theta - 4i f_t)^2 & 0 \\ 0 & (\partial_\theta + 4i f_t)^2
 \end{pmatrix}
\right).
\end{align*}
This is reduced further by restricting to the Fourier modes $\{\phi_\ell\}_{\ell\in\Z}$, leading to the family of operators
\begin{equation}\label{eq:FourierexpansionP}
P_{\ell,t}^\pm=-\frac{1}{r^2}(r\partial_r)^2+\frac{1}{r^2}(\ell\pm 4f_t)^2.
\end{equation}
As for the potential, with respect to the basis $\{\sigma_2, \sigma_3\}$, 
\[
\left. M_{\varphi^\fid_t}\right|_{iV^\perp}=8 
\begin{pmatrix}
|z|\cosh(2 h_t) + \Re z & -\Im z\\
- \Im z & |z|\cosh(2 h_t) - \Re z 
\end{pmatrix},
\]
so 
\[
\left. M^{-1} \circ M_{\varphi^\fid_t}\right|_{iV^\perp} \circ M = 8 
\begin{pmatrix}
|z| \cosh (2 h_t) & z \\
\bar z & |z| \cosh(2 h_t)
\end{pmatrix}.
\]

These calculations show that we can reduce $L_t$ to the subspaces
\[
E_\ell = \langle\varphi_\ell\sigma_2,\varphi_{\ell - 1}\sigma_3\rangle\cong L^2((0,1),rdr)\oplus L^2((0,1),rdr).
\]
To collect all these decompositions in one place, we have reductions of the standard Laplacian:
\[
P:= \Delta = \bigoplus_{\ell \in \Z} \begin{pmatrix} P_\ell & 0 \\ 0 &  P_{\ell-1} \end{pmatrix}, \qquad 
P_\ell = -\frac{1}{r^2}(r \partial_r)^2 + \frac{\ell^2}{r^2},
\]
the connection Laplacian: 
\[
P_t:=M^{-1} \circ \Delta_{A_t^\fid} \circ M = \bigoplus_{\ell \in \Z} \begin{pmatrix} P_{\ell,t}^- & 0 \\ 0 & P_{\ell-1,t}^+ \end{pmatrix},
\]
and finally $L_t = \bigoplus L_{\ell,t}$, where 
\begin{multline*}
L_{\ell,t}:=M^{-1} \circ L_t \circ M \vert_{E_\ell}
=  \begin{pmatrix}
 P_{\ell,t}^- & 0 \\ 0 & P_{\ell-1,t}^+ \end{pmatrix} +8 t^2 r \begin{pmatrix}
 \cosh (2 h_t) & 1 \\
 1 & \cosh(2 h_t).
\end{pmatrix}
\end{multline*}
The operators $L_{\ell,t}$ are self-adjoint when we impose Dirichlet boundary conditions at $r=1$ and the condition
that solutions be bounded at $r=0$. 

We now use these reductions, and the fact that $L^2(D)=\bigoplus_{\ell \in \Z} E_\ell$,  to prove the 

\begin{proposition}\label{inverse}
There exists a constant $C>0$ such that
\begin{enumerate}[1.]
\item $\|G_t \|_{\mathcal{L}(L^2,L^2)}\leq C$.
\item $\|G_t\|_{\mathcal{L}(L^2,H^2)}\leq Ct^2$.
\end{enumerate}
\end{proposition}
\begin{proof}
Let $\lambda$ denote the smallest positive eigenvalue of $P_0$. Thus
\[
\langle P_{\ell,t}^\pm \psi,\psi \rangle_{L^2} = \langle (P_0 + r^{-2}(\ell \pm 2 f_t)^2)\psi, \psi\rangle_{L^2} \geq 
\langle P_0 \psi, \psi\rangle_{L^2} \geq \lambda \|\psi\|_{L^2}^2
\]
for all $\psi \in \calC_0^\infty(0,1)$, and hence in the Friedrichs domain.

Now denote by $Q_{\ell,t}^\pm$ and $Q_t$ the inverses of $P_{\ell,t}$ and $P_t$, respectively.  We have that
$\|Q_{\ell,t}\|_{\mathcal{L}(L^2,L^2)} \leq \lambda^{-1}$ for all $\ell$ and $t$, so if $v=\sum_{\ell\in\Z}v_{\ell}\varphi_{\ell} \in L^2(B)$, then
\[
\|Q_t v\|_{L^2}^2 = \sum_{\ell \in \Z} \|Q_{\ell,t}v_\ell\|^2_{L^2} \leq  \lambda^{-2} \sum_{\ell \in \Z}  \|v_\ell\|_{L^2}^2=\lambda^{-2}\|v\|_{L^2}^2.
\]
However, $M_{\Phi_t^\fid} \geq 0$, so $P_t \leq L_t$ and therefore $\|G_t\|_{\mc L(L^2,L^2)} \leq \|Q_t\|_{\mc L(L^2,L^2)}$. This proves
the first part. 

It remains to show that $\|\Delta G_t v \|_{L^2} \leq C t^2\|v\|_{L^2}$ for all $v \in L^2(B)$. First write
\[
L_{\ell,t}-\Delta\vert_{E_\ell}=
\begin{pmatrix}
V_{\ell,t}^- & 0 \\ 0 & V_{\ell-1,t}^+
\end{pmatrix}+ W_t=:V_{\ell,t}+W_t,
\]
where 
\[
V_{\ell,t}^\pm := \frac{(\ell \pm 4f_t)^2 - \ell^2}{r^2} = \frac{16f_t^2 \pm 8\ell f_t}{r^2}, \qquad 
W_t:= 8t^2 r \begin{pmatrix}
\cosh (2 h_t) & 1 \\1 & \cosh(2 h_t)
\end{pmatrix}.
\]
Also set $G_{\ell,t}:=L_{\ell,t}^{-1}$. 

When $\ell \neq 0$, the potentials $r^{-2}( \ell \pm 2 f_t)^2$ are bounded below by $\kappa \ell^2$ for $0 < r < 1$
where $\kappa>0$ is independent of $\ell$ and $t$, cf.\ \cref{f_t-h_t-function}, and $W_t \geq 0$.  
Hence for these values of $\ell$,  
\[
\langle P_{\ell,t}^\pm \psi,\psi \rangle_{L^2} \geq \kappa \ell^2 \|\psi\|_{L^2}^2, \qquad \psi \in C_0^\infty(0,1),
\]
and so 
\begin{eqnarray}\label{eq:estGellt}
\|G_{\ell,t}\|_{\mathcal{L}(L^2,L^2)} \leq \kappa^{-1}\ell^{-2}
\end{eqnarray}

Now use \cref{f_t-h_t-function} to deduce the bounds 
\begin{eqnarray*}
\sup_{r \in (0,1)}|V_{\ell,t}^\pm(r)| \leq \begin{cases}
Ct^{4/3}, &  \quad \ell=0, \\
C\ell t^{4/3},& \quad \ell \neq 0. 
\end{cases}
\end{eqnarray*}
and 
\[
\sup_{r \in (0,1)} |W_t(r)| \leq C t^{4/3}. 
\]
Together with \eqref{eq:estGellt}, for $t\geq1$, we see that 
\begin{multline*}
\| \Delta L_t v \|_{L^2}^2 \leq \|(M^{-1}\circ L_t\circ M-\Delta)G_tv\|_{L^2}^2  + \|M^{-1}\circ L_t\circ MG_tv\|_{L^2}^2 \\[0.5ex]
= \sum_{\ell \in \Z} \|( V_{\ell,t}+W_t) G_{\ell,t} v_\ell\|
^2_{L^2} \leq Ct^4\sum_{\ell \in \Z} (1+\ell)^2 \|G_{\ell,t}v_{\ell}\|^2_{L^2} + \|v\|_{L^2}^2\\[0.5ex]
\leq Ct^4 \sum_{\ell \in \Z} (\frac{1}{\ell}+\frac{1}{\ell^2})^2 \|v_\ell\|^2_{L^2} \leq Ct^4\|v\|_{L^2}^2
\end{multline*}
where $C$ is independent of $t$. 
\end{proof}

\begin{corollary}\label{cor:estLaplinverse}
For all $u \in H^2(D)\cap H^1_0(D)$, we have $\|u\|_{H^2} \leq Ct^2\|u \|_{L_t}$,
where $\|u\|_{L_t}$ is the graph norm for the operator $L_t$. 
\end{corollary}
%
%
%
\section{Gluing construction}\label{sec:glue.const}
We are now in a position to prove the main gluing theorem. The strategy is the 
standard one: we construct a family of approximate solutions to $F_t(\gamma) = 0$, 
then use the invertibility of the linearized operator to perturb these approximate solutions
to exact solutions. 
%
\subsection{Approximate solutions}\label{appr}
Let $H(E)$ denote the bundle of Hermitian elements in $\SL(E)$.  Now consider the map
\begin{align*}
&F_t\colon H^2(H(E))\to L^2(\Lambda^2\otimes\mf{su}(E)),
\\
&\qquad \qquad F_t(g)=F_{A_{\infty}^g}^\perp+t^2[\Phi_{\infty}^g\wedge(\Phi_{\infty}^g)^*],
\end{align*}
computed at a limiting configuration $(A_{\infty},\Phi_{\infty})$. Write $X^\interior=\bigcup_{p\in\mf p}D^\times_1(p)$ for the 
union of the punctured discs, and assume that $(A_{\infty},\Phi_{\infty})$ is in fiducial form in each of these. 
To be concrete, assume that the radii are all equal to one.  We also set $X^\exterior = X \setminus \bar X^\interior$. 

Define the family of complex gauge transformations  
$$
g_t=\exp(\gamma_t), \qquad \gamma_t=\begin{pmatrix}-\tfrac{1}{2}h_t&0\\0&\tfrac{1}{2}h_t\end{pmatrix}
$$
on $X^{\interior}$; by \cref{gau.orb.fid}, 
$$
(A_t^{\fid},\Phi_t^{\fid})=(A^\fid_{\infty},\Phi^\fid_{\infty})^{g_t}
$$
on $X^\interior$. Our approximate solution is obtained by gluing $(A_t^{\fid},\Phi_t^{\fid})$ on $X^\interior$ to 
$(A_{\infty},\Phi_{\infty})$ on $X^{\exterior}$. Thus, choose a smooth cut-off function $\chi\colon X\to[0,1]$ with 
$\operatorname{supp}\chi\subseteq X^{\interior}$ and $\chi(z)\equiv1$ for $z\in\bigcup_{p\in\mathfrak p}D_{1/2}(p)$. Then
\begin{equation}\label{app.gam.t}
g_t^{\appr}(z):=\exp(\chi\gamma_t)
\end{equation}
is a family of smooth gauge transformations on $X^\times$ with
$$
g_t^\appr=g_t\mbox{ on }\bigcup_{p\in\mathfrak p}D_{1/2}(p)\mbox{ and }g_t^\appr=\Id\mbox{ on }X^\exterior.
$$ 
The new pair 
$$
(A_t^{\appr},\Phi_t^{\appr}):=(A_{\infty},\Phi_{\infty})^{g_t^{\appr}}
$$
is smooth and coincides with the fiducial solution $(A_t^{\fid},\Phi_t^{\fid})$ on  $\bigcup_{p\in\mathfrak p}D_{1/2}(p)$, and with $(A_{\infty},\Phi_{\infty})$ on $X^{\exterior}$. 

We claim that if the limiting configuration $(A_\infty,\Phi_\infty)$ is constructed from an initial pair $(A,\Phi)$, as in \cref{lim.conf}, then $(A_t^{\appr},\Phi_t^{\appr})$ 
is complex gauge equivalent to $(A,\Phi)$ by a smooth gauge transformation {\em defined over all of} $X$. Indeed, recall from \cref{lim.conf} that in a suitable 
holomorphic frame around a zero $p\in\mf p$ of $\det\Phi$, the connection matrix of $A$ vanishes and $\Phi$ is of the form of \cref{nor.for.zer}. To 
transform $(A,\Phi)$ into $(A_t^{\appr},\Phi_t^{\appr})$ we apply the gauge transformation
\[
G_t=g_\infty g_{\mu_p}g_{\mu_f}g_t^\appr
\]
where
\begin{itemize}
	\item $g_\infty$ is a normalizing gauge transformation which puts $(A,\Phi)$ into fiducial form on a neighbourhood of the zeroes of $\det\Phi$. It is obtained 
by using \cref{ext.lem} to extend the locally defined gauge transformation $g_\infty$ from \cref{gau.orb.fid} to a smooth normalizing gauge transformation on $X^\times$.
	\item $g_{\mu_p}=\exp(\gamma_{\mu_p})$ is the Hermitian gauge transformation in the stabilizer of $\Phi^\fid_\infty$ which gauges away the central part of
the curvature. This is obtained by solving the Poisson equation for $\gamma_{\mu_p}$ (cf.\ \cref{equ.s.gauge} and \cref{sol.poiss.equ}).
	\item $g_{\mu_f}=\exp(\gamma_{\mu_f})$ is the unitary gauge transformation which fiducializes $A_\infty^{g_{\mu_p}}$ (cf.\ \cref{nor.form.fid}).
	\item $g_t^\appr$ is the complex gauge transformation from~\eqref{app.gam.t}.
\end{itemize}

\begin{proposition}\label{prop:continuitysinggauge}
The complex gauge transformation $G_t$ admits a smooth extension across any point $p\in\mf p$. In particular, $(A^\appr_t,\Phi^\appr_t)$ is 
complex gauge equivalent to $(A,\Phi)$ over $X$. 
\end{proposition}
\begin{proof}
First note that we only need to prove continuity of the extension. Indeed, we can bootstrap the identity
\[
dG_t=G_tA^\appr_t-AG_t
\]
since $A^\appr_t$ and $A$ are smooth connections. Since $G_t$ is smooth on $X^\times$, the discussion is completely local. We proceed in three steps.

\begin{step}\label{coercivityD_A:step1}
The coefficient $\mu_p$ of the solution $\gamma_{\mu_p}$ (as in \cref{gamma.mu}) of the Poisson equation 
has an expansion of the form
\[
\mu_p \sim (C_0+C_1e^{-i\theta})r^{\frac{1}{2}}+O(r^{\frac{3}{2}}).
\]
\end{step}
This follows directly from the indicial root calculation for the Laplacian $\Delta_A$ in \cref{sec:ind.roots}.

\begin{step}\label{exp.muf}
The coefficient $\mu_f$ of $\gamma_f$ has
\[
\mu_f \sim (C_0-C_1e^{-i\theta})r^{\frac{1}{2}}+O(r^{\frac{3}{2}}).
\]
In particular, $\mu_p+\mu_f$ decays like $r^{\frac{1}{2}}$ as $r\to0$.
\end{step}

Indeed, $\mu_f$ is the solution of 
\[
P\mu_f:=(-i\partial_\theta+\tfrac 12)\mu_f=iv,
\]
where $v$ is the upper right entry of the $d\theta$-component of $A^{g_{\mu_p}}_\infty$ (see \cref{lim.fid.con} for the notation 
and calculations). Using the transformation formula~\eqref{traf.equ.a01} for the $(0,1)$-component of the connection shows that
\[
iv=re^{-2i\theta}D\mu_p+re^{i\theta}\overline{D\mu_p},
\]
where
\[
D=\frac{1}{2}e^{2i\theta}\big(\partial_r+\frac{i}{r}\partial_{\theta}-\frac{1}{2r}\big).
\]
Furthermore, since $\gamma_{\mu_p}$ is Hermitian, $\bar\mu=e^{i\theta}\mu$. It follows that
\[
re^{-2i\theta}D\mu+re^{i\theta}\overline{D\mu}=r\partial_r\mu
\]
so that $\mu_f$ is the solution of the ODE
\[
P\mu_f=r\partial_r\mu_p.
\]
This implies that $\mu_f$ has an expansion in powers of $r^{1/2}$ and Step~\ref{exp.muf} follows from a comparison of coefficients.

\begin{step}
We now can check continuity of the gauge transformation $G_t$ at $r=0$.
\end{step}
\setcounter{step}{0}

By \cref{gau.orb.fid} we know that
\begin{eqnarray*}
g_\infty=\begin{pmatrix}r^{\frac{1}{4}}&0\\0&r^{-\frac{1}{4}}\end{pmatrix}.
\end{eqnarray*}
Furthermore, $g_t^\appr=g^{-1}_\infty$ up to multiplication by a smooth gauge transformation, which can be ignored here. 
By Step~\ref{exp.muf}, $\mu=\mu_p+\mu_f=2C_0r^{1/2}+\mc O(r^{3/2})$, so that 
\[
g_{\mu_p}g_{\mu_f}=g_\mu=\begin{pmatrix}\cosh(e^{i\theta/2}\mu)&e^{-i\theta/2}\sinh(e^{i\theta/2}\mu)\\
e^{i\theta/2}\sinh(e^{i\theta/2}\mu)&\cosh(e^{i\theta/2}\mu)\end{pmatrix}
\]
and finally
\begin{multline*}
\begin{pmatrix}r^{\frac{1}{4}}&0\\0&r^{-\frac{1}{4}}\end{pmatrix}   \begin{pmatrix}\cosh(e^{i\theta/2}\mu)&e^{-i\theta/2}\sinh(e^{i\theta/2}\mu)\\
e^{i\theta/2}\sinh(e^{i\theta/2}\mu)&\cosh(e^{i\theta/2}\mu_2)\end{pmatrix}  \begin{pmatrix}r^{\frac{1}{4}}&0\\0&r^{-\frac{1}{4}}\end{pmatrix}\\
=\begin{pmatrix} \cosh(e^{i\theta/2}\mu)&r^{-\frac{1}{2}}e^{-i\theta/2}\sinh(e^{i\theta/2}\mu)\\
r^{\frac{1}{2}}e^{i\theta/2}\sinh(e^{i\theta/2}\mu)& \cosh(e^{i\theta/2}\mu)  \end{pmatrix}.
\end{multline*}
This is easily seen to have a limit as $r\to0$.
\end{proof}

Starting from the initial pair $(A,\Phi)$ associated with a Higgs bundle $(\delb,\Phi)$ with simple Higgs field $\Phi$, we 
have thus arrived at a complex gauge equivalent pair $(A^\appr_t,\Phi^\appr_t)$. The latter can be regarded as an approximate solution in the following sense.
\begin{lemma}
There exist $C, \delta >0$ such that for $t\gg 1$,  
\begin{equation}\label{lem:approxerror}
\|F_t(g_t^{\appr})\|_{L^2}\leq Ce^{-\delta t}.
\end{equation}
\end{lemma}
\begin{proof}
By the definition of $(A_t^{\appr},\Phi_t^{\appr})$,  it suffices to estimate the error on $X^{\interior}\setminus\bigcup_{p\in\mathfrak p}D_{1/2}(p)$. 
From the properties of $h_t$ in \cref{f_t-h_t-function} we see that $g_t$ converges to the identity on $X^{\interior}\setminus
\bigcup_{p\in\mathfrak p}D_{1/2}(p)$ like $e^{-ct}$ as $t\to\infty$. In particular, both terms on the right in 
\begin{eqnarray*}
F_t(g_t^{\appr})=F_{(A^{\infty})^{g_t^{\appr}}}^\perp+t^2[(g_t^{\appr})^{-1}\Phi_{\infty}g_t^{\appr}\wedge(g_t^{\appr})^{-1}\Phi_{\infty}g_t^{\appr})^{\ast}]
\end{eqnarray*}
converge exponentially in $t$ to $0$ (cf.\ \cref{cur.com.gt} for the curvature term). This gives~\eqref{lem:approxerror}. 
\end{proof}
%
\subsection{Global linear estimates} 
Let $L_t$ be computed at the pair $(A_t^\appr,\Phi_t^\appr)$. We now establish estimates for $G_t=L_t^{-1}\colon 
L^2(i\mf{su}(E))\to H^2(i\mf{su}(E))$. Let $\lambda_t(X) >0$ be the first eigenvalue of $L_t=\Delta_{A_t^\appr}+
t^2 M_{\Phi_t^\appr}$ on $X$, and $\lambda_t(X^{\interior})$, resp.\ $\lambda_t(X^{\exterior})$ the first Neumann 
eigenvalues of $L_t$ on $X^{\interior}$ and $X^{\exterior}$, respectively. To be clear, the domain of the 
Neumann extension on either of these regions is 
\[
\{u\in H^2(i\mf{su}(E)|_{X^{\interior/\exterior}})\mid(d_{A_t} u)\nu=0\}
\]
where $\nu$ is the unit normal $\nu$. The key result which allows us to extend the estimates above
to the whole of $X$ is the {\em domain decomposition principle}, see for instance \cite[Proposition 3]{ba00}, 
which states that
\[
\lambda_t(X) \geq \min \{ \lambda_t(X^{\interior}), \lambda_t(X^{\exterior})\}.
\]

\begin{lemma}\label{lem:firsteigenvalue}
For $t \geq 1$, there is a uniform lower bound
\[
\lambda_t(X)\geq\lambda > 0.
\]
\end{lemma}
\begin{proof}
We proceed in two steps.
\begin{step}
We have $A^\appr_t=2f_{\chi,t}\sigma_1d\theta$, where $8f_{\chi,t}=1+2r\partial(\chi h_t)$, so we can analyze 
$L_t$ via a Fourier reduction as in \cref{lin.op.fid}. Noting that $M_{\Phi^\appr_t}$ is positive on $iV$, we obtain that $L_t$ is 
strictly positive on this subbundle. On the other hand, $L_t \geq\Delta_{A^\appr_t}$ on $iV^\perp$. This requires checking that
the operator 
\[
D\varphi:=-r^{-2}(r\partial_r)^2+16r^{-2} f_{\chi,t}^2
\]
with Neumann (rather than Dirichlet) conditions at $r=1$ is strictly positive.  To see this, observe that the summands of $D$ 
are non-negative. If $L\varphi=0$, then integration by parts shows that $\partial_r \varphi=f_{\chi,t} \varphi=0$, whence $\varphi = 0$.
\end{step}

\begin{step} 
Note that $L_t\geq\Delta_{A_t^\appr} + M_{\Phi_t^\appr}$ when $t \geq 1$.  Now
\[
\int_{X^\exterior} \langle (\Delta_{A_\infty} + M_{\Phi_\infty}) \gamma, \gamma \rangle = \int_{X^\exterior} | d_{A_\infty} \gamma |^2 + 
\int_{X^\exterior} 4 | [\gamma \wedge \Phi_\infty]|^2.
\]
In particular, the kernel of the Neumann extension of $\Delta_{A_\infty} + M_{\Phi_\infty}$ consists of parallel sections $\gamma$ 
of $iL_{\Phi_\infty}$. As explained in \cref{spe.f.gau} and \cref{lim.fid.con}, this is a twisted line bundle, so $\gamma=0$. 
We conclude that this Neumann extension is invertible on $X^\exterior$, and hence has a positive first eigenvalue. Thus 
there exists $\lambda^\exterior>0$ such that
\[
\lambda_t(X^{\exterior})\geq\lambda^{\exterior}>0.
\]
\end{step}
\setcounter{step}{0}

\noindent The result now follows if we set $\lambda:=\min\{\lambda^{\interior},\lambda^{\exterior}\}$.
\end{proof}

\begin{corollary}\label{ellzwei}
$\|G_t v \|_{L^2} \leq C \| v \|_{L^2}$ for $C=\lambda^{-1}$.
\end{corollary}

We now use the $t$-dependent Sobolev space $H^2_t:=\dom L_t$, endowed with the graph norm 
\[
\| u \|^2_{L_t} = \|u\|_{L^2}^2 + \| L_t u\|_{L^2}^2.
\]
Clearly, $\|G_t v\|_{L_t} \leq C \|v\|_{L^2}$ for all $t \geq 1$ and some $C$ independent of $t$. Note that $H^2_t = H^2$ for all $t$, 
but the norms are not uniformly equivalent as $t \nearrow \infty$. 

\begin{lemma}\label{lem:equivnorms}
If $u\in H^2(i\mf{su}(E))$, then $\|u\|_{H^2}\leq Ct^2\|u\|_{L_t}$. 
\end{lemma}
\begin{proof}
Using cut-off functions, write $u=u^\interior + u^\exterior$ with $\supp u^\interior \subset X^\interior$ and 
$\supp u^\exterior \subset X \setminus\bigcup_{p\in\mathfrak p}D_{1/2}(p)$. Then by \cref{cor:estLaplinverse} we have
\[
\|u^\interior\|_{H^2} \leq C (1+t^2)\|u\|_{L_t}.
\]
On $X \setminus\bigcup_{p\in\mathfrak p}D_{1/2}(p)$, consider the linear operator
\[
\tilde{L}_t:=\Delta_{A_\infty} + t^2 M_{\Phi_\infty}
\]
with Dirichlet boundary conditions. Then $\tilde L_t$ is invertible and we write $\tilde{G}_t:=\tilde{L}_t^{-1}$. 
Now 
\[
\|\tilde L_t u\|_{L^2}\leq \| L_t u\|_{L^2} + \|(\tilde L_t - L_t )u\|_{L^2}
\]
and since $A_t$ converges to $A_\infty$ and $\Phi_t$ converges to $\Phi_\infty$ exponentially in $t$, 
\[
\| (\tilde{L}_t - L_t)u\|_{L^2} \leq C e^{-\delta t} \|u\|_{L^2}.
\]
In addition, 
\begin{align*}
\| \Delta_{A_\infty} u\|_{L^2} &=\|\Delta_{A_\infty}u + t^2 M_{\Phi_\infty} u - t^2 M_{\Phi_\infty} u\| \\
&\leq \|\tilde{L}_t u \|_{L^2} + t^2 \|M_{\Phi_\infty}u\|_{L^2}\\
& \leq \|\tilde{L}_t u \|_{L^2} + t^2 \sup |M_{\Phi_\infty}| \|u\|_{L^2},
\end{align*}
which leads to the estimate
\[
\|\Delta_{A_\infty}u\|_{L^2} \leq \| L_t u\|_{L^2} + C_t \|u\|_{L^2} +Ct^2 \|u\|_{L_2}.
\]
This gives the claim since the graph norm of $\Delta_{A_\infty}$ is equivalent to the standard $H^2$-norm .
\end{proof}

Summarizing we proved the following global linear estimate.

\begin{proposition}\label{lem:globallinest}
Let $(A_t^\appr,\Phi_t^\appr)$ be the approximate solution from \cref{appr}. Then the inverse $G_t$ to 
$L_t=\Delta_{A_t^\appr}+t^2 M_{\Phi_t^\appr}$ satisfies 
\[
\|G_t v\|_{H^2}\leq Ct^2\|v\|_{L^2}.
\]
\end{proposition}
%
\subsection{Deforming the approximate solutions}
We are now finally prepared to give the argument which shows how to perturb the approximate solutions $(A_t^{\appr},\Phi_t^{\appr})$ 
to an exact solution of Hitchin's equations when $t \gg 1$. 

\begin{theorem}\label{thm:nonlinearbvaluedisk}
Let $B_\rho$ be the closed ball of radius $\rho$ around the zero section in $H^2(i\mf{su}(E))$. Then there is a value $m > 0$ and 
a unique Hermitian $\gamma_t\in B_{t^{-m}}$ such that, when $t$ is sufficiently large, 
$(A_t,\Phi_t):=(A_t^{\appr},\Phi_t^{\appr})^{\exp(\gamma_t)}$ solves the rescaled Hitchin equations. 
\end{theorem}

\begin{remark}
\cref{thm:nonlinearbvaluedisk} gives a solution to the original Hitchin equations for the Higgs bundle $(\delb,t\Phi)$, when the parameter $t$ is large, which 
is complex gauge equivalent to the initial pair $(A,t\Phi)$ as shown by \cref{prop:continuitysinggauge}. In this way, \cref{thm:nonlinearbvaluedisk} provides 
a constructive proof of Hitchin's existence theorem (when $t\Phi$ is large). We can regard \cref{thm:nonlinearbvaluedisk} as a desingularization theorem 
for limiting configurations. This shows in particular that any limiting configuration arises from a Higgs bundle. In this way we can think of the real 
$6\gamma-6$-dimensional torus of limiting configurations from \cref{bound.stratum} as a boundary stratum of Hitchin's moduli space obtained by 
projectivizing the fibre $\det^{-1}(q)$ for a fixed determinant $q\in H^0(X,K^2)$ with simple zeroes. 
\end{remark}

The solution $\gamma_t$ is obtained using a standard contraction mapping argument. To do this, we study the linearization
$L_t$,  computed at $(A_t^{\appr},\Phi_t^{\appr})$.  The argument relies on controlling the following quantities: 
\begin{itemize}
\item the norm of the inverse $L_t^{-1}$, and
\item the Lipschitz constants of the linear and higher order terms in the Taylor expansion of $F_t$.
\end{itemize}
The first of these was handled by \cref{lem:globallinest}, but we must now study the nonlinear terms in $F_t$ in greater detail. 
 
\medskip

For $g=\exp(\gamma)$, $\gamma\in\Omega^0(i\mf{su}(E))$, we have
\[
\mathcal{O}_{(A,\Phi)}(g)=(A,\Phi)^g=(A+g^{-1}(\bar\partial_{A}g)-(\partial_{A}g)g^{-1},g^{-1}\Phi g),
\]
and consequently, 
\begin{gather*}
A^{\exp\gamma}=A + (\bar\partial_{A} - \partial_{A}) \gamma + R_{A}(\gamma)\\
\Phi^{\exp\gamma}=\Phi + [\Phi, \gamma] + R_{\Phi}(\gamma).
\end{gather*}
The explicit expressions of these remainder terms are
\begin{equation}
\label{eq:remainderRA}
R_{A}(\gamma)=\exp(-\gamma) (\bar\partial_{A}(\exp \gamma))-(\partial_{A}(\exp \gamma)) 
\exp(-\gamma)- (\bar\partial_{A}-\partial_{A}) \gamma
\end{equation}
\begin{equation}\label{eq:remainderRPhi}
R_{\Phi}(\gamma)= \exp(-\gamma) \Phi \exp \gamma - [\Phi,\gamma]-\Phi.
\end{equation}
We then calculate that 
\begin{equation}
\begin{array}{rl}
F_t(\exp \gamma) &= F^\perp_{(A_t^\appr)^{\exp(\gamma)}}+t^2[(\Phi_t^\appr)^{\exp(\gamma)}\wedge(\Phi_t^\appr)^{\exp(\gamma)})^\ast] \\
&=\pr_1\mathcal{H}_t(A_t^\appr,\Phi_t^\appr)+ L_t \gamma + Q_t(\gamma)
\end{array}
\label{eq:TaylorexpF}
\end{equation}
where, in full detail, 
\begin{align*}
Q_t(\gamma) =& d_{A_t^\appr}(R_{A_t^\appr}(\gamma))+t^2[R_{\Phi_t^\appr}(\gamma)\wedge(\Phi_t^\appr)^*]+ t^2[\Phi_t^\appr\wedge R_{\Phi_t^\appr}(\gamma)^*]\\ 
&+\frac 12 [((\bar\partial_{A_t^\appr}-\partial_{A_t^\appr})\gamma+R_{A_t^\appr}(\gamma))\wedge ((\bar\partial_{A_t^\appr}-\partial_{A_t^\appr})\gamma + R_{A_t^\appr}(\gamma))]\\
&+t^2[([\Phi_t^\appr,\gamma]+R_{\Phi_t^\appr}(\gamma))\wedge([\Phi_t^\appr,\gamma]+R_{\Phi_t^\appr}(\gamma))^*].
\end{align*}

\begin{lemma}\label{lem:growthnormA}
The approximate solution satisfies 
\[
\|A^\appr_t\|_{C^1} \leq C t
\]
on the disk $D_1(0)$, so that for any $H^{k+1}$ section $\gamma$, $k = 0, 1$,
\[
\|d_{A^\appr_t}\gamma\|_{H^k}\leq Ct\|\gamma\|_{H^{k+1}},
\]
and moreover, 
\[
\|L_t \gamma \|_{L^2} \leq C t^2 \| \gamma \|_{H^2}.
\]
\end{lemma}
\begin{proof}
We have
\[
A^\appr_t=f_{\chi,t}\begin{pmatrix} i & 0 \\ 0 & -i \end{pmatrix} d\theta,
\]
where $f_{\chi,t}(r)=\tfrac 18+\tfrac 14 r\partial_r(\chi h_t)(r)$, see \cref{appr}. Clearly $f_{\chi,t}$ has the same asymptotics 
as $f_t$; thus $f_{\chi,t}$ is uniformly bounded in $t$. 

Now recall from~\cref{des.fid.sol} that $f_t(r)=\eta(\rho)$, $\rho=\tfrac{8t}{3}r^{3/2}$, where $\eta(\rho)=\frac 18 + 
\frac 38\rho \psi'(\rho)$. Then 
\[
\partial_r f_t(r) = 4tr^{1/2}\eta'(\rho),
\]
and we already know that $\eta'(\rho)=\frac{3}{16}\rho\sinh(\psi(\rho))$. Since $\psi(\rho)\sim-\log\rho$ as 
$\rho\to0$ and $\psi(\rho)\sim e^{-\rho}$ as $\rho\to\infty$, we see that 
$\lim_{\rho\to0}\eta'(\rho)=\lim_{\rho\to\infty} \eta'(\rho)=0$. This gives that $0\leq\eta'(\rho)\leq C_0$ for some constant 
$C_0>0$. Altogether, 
\[
|\partial_r f_t|=|\tfrac 14\partial_r h_t+\tfrac{r}{4}\partial_r^2h_t|\leq C_1t
\]
for some constant $C_1>0$ which also yields the desired estimate for $|\partial_r f_{\chi,t}|$.
\end{proof}

\begin{lemma}
There exists a constant $C>0$ such that 
\begin{equation}\label{eq:Lipschitzhigherorder}
\|Q_t(\gamma_1)-Q_t(\gamma_0)\|_{L^2}\leq C \rho t^2\|\gamma_1-\gamma_0\|_{H^2}
\end{equation}
for all $0<\rho\leq1$ and  $\gamma_0$, $\gamma_1\in B_{\rho}$.
\end{lemma}
\begin{proof}
The proof has two steps. To simplify notation, write $(A,\Phi)$ for $(A_t^\appr, \Phi_t^\appr)$.

\begin{step}\label{step:remain.terms}
We first check that if $\rho \in (0,1]$ and $\gamma_0, \gamma_1\in B_{\rho}$, then 
\begin{align*}
&\|R_A(\gamma_1)-R_A(\gamma_0) \|_{H^1}\leq Ct\rho\|\gamma_1 -\gamma_0\|_{H^2}\\
&\|R_\Phi(\gamma_1) - R_\Phi(\gamma_0) \|_{H^1} \leq Ct\rho\|\gamma_1 -\gamma_0\|_{H^2}. 
\end{align*}
We begin by estimating the difference of the first two terms on the right in \eqref{eq:remainderRA}: 
\begin{multline*}
\big\|\exp(-\gamma_1) (\bar\partial_A(\exp \gamma_1))-\exp(-\gamma_0)(\bar\partial_A(\exp \gamma_0))-\bar\partial_A 
(\gamma_1-\gamma_0) \big\|_{H^1} \\[0.5em]
\leq \| (\exp(-\gamma_1)-\exp(-\gamma_0))\bar\partial_A(\exp(\gamma_1))\|_{H^1} \\[0.5em]
+\|\exp(-\gamma_0)\big(\bar\partial_A(\exp(\gamma_1)-\exp(\gamma_0))\big)-\bar\partial_A(\gamma_1-\gamma_0)\|_{H^1} 
:= \mathrm{I} + \mathrm{II}.
\end{multline*}
Writing $\exp(\gamma)=1+\gamma+S(\gamma)$, then we have 
\begin{align*}
\|\mathrm{I}\|_{H^1}&\leq C_0\|\exp(-\gamma_1)-\exp(-\gamma_0)\|_{H^2}\|\bar\partial_A(\exp(\gamma_1))\|_{H^1}\\
&\leq C_1t\|\gamma_1-\gamma_0\|_{H^2}\|\gamma_1+S(\gamma_1)\|_{H^2}\\
&\leq C_2t\rho\|\gamma_1-\gamma_0\|_{H^2},
\end{align*}
and similarly, 
\begin{align*}
\|\mathrm{II}\|_{H^1}=&\,\|(1-\gamma_0+S(-\gamma_0))\big(\bar\partial_A(\gamma_1-\gamma_0+S(\gamma_1)-S(\gamma_0)\big)-\bar\partial_A(\gamma_1-\gamma_0)\|_{H^1}\\
\leq&\,\|\bar\partial_A(S(\gamma_0)-S(\gamma_1))\|_{H^1}\\
&+\|(-\gamma_0+S(-\gamma_0))\bar\partial_A(\gamma_0-\gamma_1+S(\gamma_0)-S(\gamma_1))\|_{H^1}\\
\leq&\, C_0t\|S(\gamma_0)-S(\gamma_1)\|_{H^2}\\
&+C_0t\| -\gamma_0+S(-\gamma_0)\|_{H^2}\|\gamma_0-\gamma_1+S(\gamma_0)-S(\gamma_1)\|_{H^2}\\
\leq&\,C_1t\rho\|\gamma_1-\gamma_0\|_{H^2},
\end{align*}
where we have estimated $\|S(\gamma_0)-S(\gamma_1)\|_{H^2}\leq\|\gamma_0-\gamma_1\|_{H^2}\sum_{k\geq1}\rho^k/k! 
\leq C\rho\|\gamma_0-\gamma_1\|_{H^2} $. These estimates together with analogous ones for the terms involving 
$\partial_{A_t}$ give the stated Lipschitz estimate for $R_{A_t}$. The corresponding estimate for 
\[
R_\Phi= \exp(-\gamma)\Phi\exp\gamma-[\Phi,\gamma]-\Phi 
\]
and the estimates
\[
\|R_A(\gamma)\|_{H^1}\leq Ct\rho,\quad\|R_\Phi(\gamma)\|_{H^1}\leq C\rho, \qquad \gamma\in B_\rho, 
\]
follow in the same way. 
\end{step}

\begin{step}
We can now prove the claim. First,
\begin{equation}
\label{eq:LipschitzQ}
\begin{aligned}
Q_t(\gamma_1) - & Q_t(\gamma_0) = d_{A} (R_{A}(\gamma_1) - R_{A}(\gamma_0))\\[0.5ex]
&+ t^2 [ (R_{\Phi}(\gamma_1) - R_{\Phi}(\gamma_0)) \wedge \Phi^*] + t^2 [\Phi \wedge (R_{\Phi}(\gamma_1) - R_{\Phi}(\gamma_0))^*]\\[0.5ex]
&+ \tfrac 12 [((\bar\partial_{A}-\partial_{A})\gamma_1 + R_{A}(\gamma_1)) \wedge ((\bar\partial_{A}-\partial_{A})\gamma_1 + 
R_{A}(\gamma_1))]\\[0.5ex]
&- \tfrac 12 [((\bar\partial_{A}-\partial_{A})\gamma_0 + R_{A}(\gamma_0)) \wedge ((\bar\partial_{A}-\partial_{A})\gamma_0 + 
R_{A}(\gamma_0))]\\[0.5ex]
&+t^2 [([\Phi, \gamma_1] + R_{\Phi}(\gamma_1)) \wedge ([\Phi, \gamma_1] + R_{\Phi}(\gamma_1))^* ]\\[0.5ex]
&-t^2 [([\Phi, \gamma_0] + R_{\Phi}(\gamma_0)) \wedge ([\Phi, \gamma_0] + R_{\Phi}(\gamma_0))^* ]. 
\end{aligned}
\end{equation}
By \cref{lem:growthnormA}, 
\[
\|d_{A_t} (R_{A_t}(\gamma_1) - R_{A_t}(\gamma_0))\|_{L^2}\leq C(t+1)\|  R_{A_t}(\gamma_1) - R_{A_t}(\gamma_0)\|_{H^1}
\]
and we then apply Step~\ref{step:remain.terms}. The remaining terms are bilinear combinations $B(\psi,\tau)$ of functions 
$\psi$ and $\tau$ with fixed coefficients, which can be estimated as 
\begin{multline*}
\|B(\psi_1,\tau_1)-B(\psi_0,\tau_0)\|_{L^2}\leq\,\|B(\psi_1-\psi_0,\tau_1)\|_{L^2}+\|B(\psi_0,\tau_1-\tau_0)\|_{L^2}\\
\leq \,C\|\psi_1-\psi_0\|_{H^1}\|\tau_1\|_{H^1} +C\|\psi_0\|_{H^1}\|\tau_1-\tau_0\|_{H^1}.
\end{multline*}
\end{step}
The desired estimate follows from Step~\ref{step:remain.terms} again. 
\end{proof}

\begin{proof}[Proof of \cref{thm:nonlinearbvaluedisk}] From \eqref{eq:TaylorexpF}, 
$$
F_t(\exp(\gamma))=\pr_1\mathcal{H}_t(A_t^{\appr},\Phi_t^{\appr})+L_t\gamma+Q_t(\gamma),
$$
and since $L_t$ is invertible, the solutions of this equation are the same as the solutions of 
\[
\gamma=-L_t^{-1}\big(\pr_1\mathcal{H}_t(A_t^{\appr},\Phi_t^{\appr})+Q_t(\gamma)\big).
\]
Thus consider the map 
\[
T\colon B_\rho\to H^2(i\mf{su}(E)),\quad\gamma\mapsto -L_t^{-1}\big(\pr_1\mathcal{H}_t(A_t^{\appr},\Phi_t^{\appr})+Q_t(\gamma)\big).
\]
We claim that for $\rho$ sufficiently small, $T$ is a contraction of $ B_\rho$, from which we immediately obtain
a unique fixed point $\gamma\in B_\rho$. To prove this, use \cref{lem:globallinest} and \eqref{eq:Lipschitzhigherorder} to get 
\begin{multline*}
\|T(\gamma_1-\gamma_0)\|_{H^2} =\|G_t(Q_t(\gamma_1)-Q_t(\gamma_0))\|_{H^2}\\
\leq Ct^2\|Q_t(\gamma_1)-Q_t(\gamma_0)\|_{L^2} \leq C\rho t^4\|\gamma_1-\gamma_0\|_{H^2}.
\end{multline*}
Thus $T$ is a contraction on the ball of radius $\rho_t=t^{-4-\epsilon}$ for any $\epsilon > 0$. Furthermore, since $Q_t(0)=0$, 
then by \cref{lem:globallinest} and \eqref{lem:approxerror}, 
\[
\|T(0)\|_{H^2}=\|G_t(\pr_1\mathcal{H}_t(A_t^{\appr},\Phi_t^{\appr})\|_{H^2}\leq C_te^{-\delta t}. 
\]
Thus when $t \gg 0$, $\|T(0)\|_{H^2} < \frac{1}{10}\rho_t $, so the ball $B_{\rho_t}$ is mapped to itself by $T$. 
\end{proof}
%
%
%

\end{document}